\documentclass[10pt]{article}
\usepackage{fullpage}
\usepackage{graphicx}
\usepackage{amsmath, amssymb, amsthm, mathtools}
\usepackage{nicefrac,tcolorbox}
\usepackage{verbatim}
\usepackage{enumitem, xcolor}
\usepackage{url}
\usepackage{authblk}
\usepackage[colorlinks,linkcolor={blue}]{hyperref}

\newcounter{listacnt}\renewcommand{\thelistacnt}{\alph{listacnt}}

\definecolor{lightblue}{rgb}{0.8,0.8,1}  

\newcommand{\cc}{\mathbb{C}}
\newcommand{\dd}{\mathbb{D}}

\newcommand{\nn}{\mathbb{N}}

\newcommand{\cF}{{\mathcal F}}

\newcommand{\cL}{{\mathcal L}}

\newcommand{\cP}{{\mathcal P}}

\newcommand{\cS}{{\mathcal S}}

\newcommand{\overbar}[1]{\mkern 1.5mu\overline{\mkern-1.5mu#1\mkern-1.5mu}\mkern 1.5mu}

\DeclarePairedDelimiter{\setof}{\{}{\}}  
\DeclarePairedDelimiter{\abs}{\lvert}{\rvert}  
\DeclarePairedDelimiter{\norm}{\lVert}{\rVert}
\DeclarePairedDelimiter{\dotp}{\langle}{\rangle}

\renewcommand{\Re}{\text{Re}}
\renewcommand{\Im}{\text{Im}}
\renewcommand{\forall}{\text{ for all }}

\theoremstyle{plain} 
\newtheorem{theorem}{Theorem} 
\newtheorem{corollary}[theorem]{Corollary}

\newtheorem{lemma}[theorem]{Lemma}
\newtheorem{proposition}[theorem]{Proposition}
\newtheorem{hyp}[theorem]{Hypothesis}

\theoremstyle{definition}
\newtheorem{definition}[theorem]{Definition}
\newtheorem{example}[theorem]{Example}

\theoremstyle{remark}
\newtheorem{remark}[theorem]{Remark}

\makeatletter
\def\renewtheorem#1{%
	\expandafter\let\csname#1\endcsname\relax
	\expandafter\let\csname c@#1\endcsname\relax
	\gdef\renewtheorem@envname{#1}
	\renewtheorem@secpar
}
\def\renewtheorem@secpar{\@ifnextchar[{\renewtheorem@numberedlike}{\renewtheorem@nonumberedlike}}
\def\renewtheorem@numberedlike[#1]#2{\newtheorem{\renewtheorem@envname}[#1]{#2}}
\def\renewtheorem@nonumberedlike#1{  
	\def\renewtheorem@caption{#1}
	\edef\renewtheorem@nowithin{\noexpand\newtheorem{\renewtheorem@envname}{\renewtheorem@caption}}
	\renewtheorem@thirdpar
}
\def\renewtheorem@thirdpar{\@ifnextchar[{\renewtheorem@within}{\renewtheorem@nowithin}}
\def\renewtheorem@within[#1]{\renewtheorem@nowithin[#1]}
\makeatother
\def\corcommstyle{\bf\small\tt}

\def\corrl #1<<#2||#3>>{
\if\visiblecomments y
  \begin{quote} {\corcommstyle $<<$COMMENT$>>$ {\color{red}#1\marginpar{!!}}\\$<<$OLD$<<$} \end{quote}

{\color{red} 
 #2
 }

  \begin{quote} {\corcommstyle ==NEW== } \end{quote}
   \noindent\hrulefill
 
\vspace{-10pt} 
 
 \noindent\hrulefill
 
 \vspace{-10pt} 
 
 \noindent\dotfill
 
  #3
  
   \noindent\dotfill 

\vspace{-10pt} 
 
 \noindent\hrulefill
 
 \vspace{-10pt} 
 
 \noindent\hrulefill
  \begin{quote} {\corcommstyle $>>$END$>>$ } \end{quote}
 \else
  #3
 \fi
}

\long\def\longcorrl #1<<#2||#3>>{
\if\visiblecomments y
  \begin{quote} {\corcommstyle $<<$COMMENT$>>$ {\color{red}#1\marginpar{!!}}\\$<<$OLD$<<$} \end{quote}
 
 {\color{red}

  #2
  
  }
  
  \begin{quote} {\corcommstyle ==NEW== } \end{quote}
  
    \noindent\hrulefill
 
\vspace{-10pt} 
 
 \noindent\hrulefill
 
 \vspace{-10pt} 
 
 \noindent\dotfill
 
  #3
  
   \noindent\dotfill 

\vspace{-10pt} 
 
 \noindent\hrulefill
 
 \vspace{-10pt} 
 
 \noindent\hrulefill
  \begin{quote} {\corcommstyle $>>$END$>>$ } \end{quote}
 \else
  #3
 \fi
}

\def\mlabel #1
{
  \if\visiblecomments y
     \marginpar[\flushright \bf \footnotesize #1]{\bf \footnotesize #1}
  \fi
}

\def\flabel #1
{
  \if\visiblecomments y
       \hbox{\bf\footnotesize #1}
  \fi
}

\def\corrq #1<<#2>>{
\if\visiblecomments y
  \begin{quote} {\corcommstyle $<<$COMMENT$>>$ {\color{red}#1}\marginpar{!!}\\$<<$BEG$<<$} \end{quote}
  \noindent\hrulefill
 
\vspace{-10pt} 
 
 \noindent\hrulefill
 
 \vspace{-10pt} 
 
 \noindent\dotfill

  #2
 
  \noindent\dotfill 

\vspace{-10pt} 
 
 \noindent\hrulefill
 
 \vspace{-10pt} 
 
 \noindent\hrulefill 
  \begin{quote} {\corcommstyle $>>$END$>>$ } \end{quote}
 \else
  #2
 \fi
}

\long\def\longcorrq #1<<#2>>{
\if\visiblecomments y
  \begin{quote} {\corcommstyle $<<$COMMENT$>>$ #1\marginpar{!!}\\$<<$BEG$<<$} \end{quote}
  \noindent\hrulefill
 
\vspace{-10pt} 
 
 \noindent\hrulefill
 
 \vspace{-10pt} 
 
 \noindent\dotfill

  #2

  \noindent\dotfill 

\vspace{-10pt} 
 
 \noindent\hrulefill
 
 \vspace{-10pt} 
 
 \noindent\hrulefill 
  \begin{quote} {\corcommstyle $>>$END$>>$ } \end{quote}
 \else
  #2
 \fi
}

\def\corrc #1<<>>{
\if\visiblecomments y
  \begin{quote} {\corcommstyle $<<$COMMENT$>>$ \color{red} #1\marginpar{!!}} \end{quote}
\fi
}

\def\corre #1<<#2||#3>>{
\if\visiblecomments y
  #3\marginpar{\corcommstyle #1}
 \else
  #3
 \fi
}

\long\def\longcorre #1<<#2||#3>>{
\if\visiblecomments y
  #3\marginpar{\corcommstyle #1}
 \else
  #3
 \fi
}

\def\corrs #1<<#2||#3>>{
\if\visiblecomments y
  #3\marginpar{\corcommstyle #2 $\rightarrow$ #3\\ #1}
 \else
  #3
 \fi
}

\def\corro #1<<#2||#3>>{
#2}

\def\corrn #1<<#2||#3>>{
#3}

\long\def\longcorro #1<<#2||#3>>{
#2}

\long\def\longcorrn #1<<#2||#3>>{
#3}

\long\def\underconstruction #1<<<#2>>>{
\if\visiblecomments y
  \begin{quote} {\corcommstyle $<<$UNDER CONSTRUCTION - BEGIN$>>$ #1\marginpar{!!}} \end{quote}
  #2
  \begin{quote} {\corcommstyle $>>$UNDER CONSTRUCTION - END$>>$ } \end{quote}
 \else
 \fi
}

\def\showcomments{
  \let\visiblecomments y
}

\def\hidecomments{
  \let\visiblecomments n
}

\showcomments

\newcommand{\numder}{\overbar{A}^\dagger}
\newcommand{\numinv}{\overbar{A}}

\newcommand{\pode}{\mathcal{P}}
\newcommand{\pdde}{P}
\newcommand{\pps}{P_n}

\newcommand{\hoode}{\mathcal{H}}
\newcommand{\hodde}{H}
\newcommand{\hops}{H_n}

\newcommand{\zeroode}{\mathcal{F}}
\newcommand{\zerodde}{F}
\newcommand{\zerops}{F_n}

\newcommand{\linps}{\mathcal{A}_n}
\newcommand{\nonlinps}{\mathcal{G}_n}

\newcommand{\linode}{\mathcal{A}}
\newcommand{\nonlinode}{\mathcal{G}}

\newcommand{\blowupscalardde}{r}
\newcommand{\blowupscalarps}{r_n}
\newcommand{\blowupvectorps}{R_n}

\newcommand{\placen}{10}
\newcommand{\placeeps}{\Re(\lambda_n)M - \norm*{D}}
\newcommand{\placeN}{25}
\newcommand{\placeM}{1000}
\newcommand{\placealpha}{2}

\newcommand{\github}{\href{https://github.com/skepley/pseudospectral\_DDE\_CAP}{Github}}

\begin{document}
\title{Validated error bounds for pseudospectral approximation of delay differential equations: unstable manifolds}
\author[1]{Shane Kepley} 
\author[1]{Babette A.J. de Wolff}
\affil[1]{Department of Mathematics, Vrije Universiteit Amsterdam, Boelelaan 1111, 1081 HV Amsterdam, The Netherlands}
\date{}
\maketitle
\tableofcontents
\newpage 

\begin{abstract}
Pseudospectral approximation provides a means to approximate the dynamics of delay differential equations (DDE) by ordinary differential equations (ODE). This article develops a computer-aided algorithm to determine the distance between the unstable manifold of a DDE and the unstable manifold of the approximating pseudospectral ODE. The algorithm is based upon the parametrization method. While a-priori the parametrization method for a vector-valued ODE involves computing a sequence of vector-valued Taylor coefficients, we show that for the pseudospectral ODE, due to its specific structure, the problem reduces to finding a sequence of scalars, which significantly simplifies the problem. 
\end{abstract}

\section{Introduction}
  
Numerical bifurcation analysis is an indispensable tool in contemporary applied mathematics and applied sciences. 
For dynamical systems that are generated by delay differential equations (DDE), the software packages DDE-Biftool \cite{biftool} and KNUT \cite{knut} provide extensive
tools that are based upon direct analysis
of the infinite dimensional dynamical system that is generated by a DDE. 
In recent years, an alternative two-step approach to numerical bifurcation analysis of DDE has been advocated \cite{breda16}. The first step of this approach is to systematically approximate the dynamics of a DDE by a finite dimensional ordinary differential equations (ODE); next one numerically analyses the resulting ODE using existing ODE bifurcation software (such as Auto \cite{auto} or MatCont \cite{matcont}). 
In \cite{breda16, breda05}, the approach to deriving the `approximating ODE' is based upon polynomial interpolation and goes under the name pseudospectral approximation. 

Pseudospectral approximation as a method for numerical bifurcation analysis has had many applications over the past years, see e.g. \cite{Zhang22, Breda21}.  
In order to also provide an analytical
foundation to the approach, one should verify that the bifurcation behaviour of the DDE is well approximated by the bifurcation behaviour of the pseudospectral ODE. Results on convergence of eigenvalues \cite{breda05} and on the Hopf bifurcation \cite{dW21} provide the first steps in this direction. 
A natural next step 
is to investigate the approximation of invariant manifolds, i.e. to determine error bounds between the invariant manifolds of the original DDE and invariant manifolds of its pseudospectral approximation. In this article, we progress in this direction by providing a computer-assisted algorithm to determine the distance between an unstable manifold of a DDE and the unstable manifold of an approximating pseudospectral ODE. In the spirit of computer assisted proofs, we do \emph{not} provide an asymptotic statement (i.e. a statement that is valid when the discretisation index goes to infinity), but rather provide a constructive method to rigorously obtain an error bound for a concrete DDE with a fixed discretization index.  

\medskip

The use of computer assisted proofs in dynamical systems has seen a rapid development in recent years. 
Among many results (see e.g. the survey \cite{vdBerg15}), we highlight the proof of Wright's conjecture by computer-aided methods \cite{vdBerg18} as a recent and prominent case in the field of delay equations.  
In the general framework of computer assisted proofs in dynamical systems,
the starting point 
is to recast the dynamical problem under consideration (in our case, the problem of finding an unstable manifold in the equation of interest) as a zero-finding problem, i.e. to construct a map $F: X \to X$ on a Banach space $X$ such that the dynamical object of interest corresponds to a root of the equation
\begin{equation}\label{eq:zero_intro}
F(x) = 0. 
\end{equation}
Using numerical methods, one then computes an approximate solution of \eqref{eq:zero_intro}, i.e. an object $\hat{x} \in X$ such that $F(\hat{x})$ is small. 
Next one proves that there is an exact solution of \eqref{eq:zero_intro} in a neighbourhood of the approximate solution $\hat{x}$. This is typically done by constructing
an invertible operator $A: X \to X$ and a number $r > 0$ such that 
Newton-like function
\begin{equation} \label{eq:newton_intro}
T(x) = x - A F(x)
\end{equation}
is a contraction on the ball $B(\hat{x}, r)$. The contraction mapping principle then implies that the equation \eqref{eq:zero_intro} has an exact solution $x_\ast$ on $B(\hat{x}, r)$, which represents our dynamical object of interest. Since $\norm{\hat{x} - x_\ast} < r$, we also automatically obtain an error bound between the numerically computed solution and the exact solution. 

\medskip 

In this paper, we will recast the problem of finding a local unstable manifold (either of a DDE or of the approximating ODE) as a zero finding problem by means of the \emph{parametrization method} \cite{Cabre03, Cabre05}.  
As a first step, 
we write the local unstable manifold of the dynamical system under consideration as a graph of a map
\[
P(\sigma): \mathbb{C}^d \to Y,
\]
where $d \in \mathbb{N}$ is the unstable dimension of the equilibrium and $Y$ is the state space of the dynamical system; we then formally expand the map $P$ as  
\begin{equation} \label{eq:Taylor_intro}
P(\sigma) = \sum_{\beta \in \mathbb{N}^d} P_\beta \sigma^\beta.
\end{equation}
The parametrization method then provides a way to derive recursive formula's for the coefficients $P_\beta$, essentially by conjugating the dynamics on the unstable manifold to the dynamics on the state space (we will review this in detail in Section \ref{sec:pm_ode_dde}). 

For a DDE with physical dimension one 
and maximal delay $\tau > 0$, the state space is typically chosen to be $C([-\tau, 0], \mathbb{R})$, so in this case each coefficient $P_\beta$ in \eqref{eq:Taylor_intro} is 
a \emph{function} $P_\beta: [-\tau, 0] \to \mathbb{R}$.
In \cite{Groothedde17} the authors showed that the function $P_\beta$ always is of the form 
\begin{equation} \label{eq:P_intro}
P_\beta(\theta) = p_\beta e^{\dotp{\lambda, \beta} \theta}, 
\end{equation}
where $\lambda \in \mathbb{C}^d$ is a vector containing the unstable eigenvalues of the DDE. 
One can then derive recursive formulas for the \emph{scalar} coefficients $p_\beta$ by means of the parametrization method, and in this way construct a map
\begin{equation}
F: \mathcal{S}_d \to \mathcal{S}_d \label{eq:zero_dde_intro}
\end{equation}
on the sequence space
$\mathcal{S}_d = \{ (x_\beta)_{\beta \in \mathbb{N}^d} \mid x_\beta \in \mathbb{R} \}
$
so that if the sequence $p  = (p_\beta)_{\beta \in \mathbb{N}^d}$ is a zero of the map $F$, then the (function-valued) map
\[
P(\sigma) = \sum_{\beta \in \mathbb{N}^d} p_\beta e^{\dotp{\lambda, \beta} .} \sigma^\beta
\]
gives a local parametrization of the unstable manifold of the DDE. 

In this article, we show that for the pseudospectral approximation of a DDE, the situation is very similar. Given a discretization index $n \in \mathbb{N}$, the pseudospectral approximation to a DDE is an ODE with state space $\mathbb{R}^{n+1}$,  
and hence each of the coefficients $P_\beta$ in \eqref{eq:Taylor_intro} is a $n+1$ dimensional vector. 
We will show that, due to the very specific form of the pseudospectral ODE, the vectors $P_\beta$ are always of the form 
\[
P_\beta = p_\beta v_{\beta}. 
\]
Here each $v_\beta$ is an a-priori known vector, whose explicit expression only depends on the unstable eigenvalues of the pseudospectral ODE. We can next derive recursive formula's for the \emph{scalar} coefficients $p_\beta$, which leads to 
the following (informally formulated) statement:

\begin{theorem}[cf. Theorem \ref{prop:parametrization_ps} in Section \ref{sec:parametrization ps}] \label{thm:thm_intro}
For a fixed number $n \in \mathbb{N}$, suppose that the $n$-th order pseudospectral approximation of an ODE has unstable dimension $d$ and let $\mathcal{S}_d = \{ (x_\beta)_{\beta \in \mathbb{N}^d} \mid x_\beta \in \mathbb{C} \}$. Then there exists a map 
\begin{equation} \label{eq:zero_ode_intro}
F_n: \mathcal{S}_d \to \mathcal{S}_d
\end{equation}
and a sequence of vectors 
\[
\{ v_\beta \mid \beta \in \mathbb{N}^d \}
\]
so that if $p  = (p_\beta)_{\beta \in \mathbb{N}^d}$ is a sequence that satisfies
\[
F_n(p) = 0 
\]
then the map 
\[
P(\sigma) = \sum_{\beta \in \mathbb{N}^d} p_\beta v_\beta \sigma^\beta
\]
is a parametrization of the local unstable manifold of the pseudospectral ODE. 
\end{theorem}

Since the maps $F$ (cf. \eqref{eq:zero_dde_intro}) and $F_n$ (cf. \eqref{eq:zero_ode_intro}) are defined on the same state space, we can make sense of the distance between their respective zeroes and interpret this as the `distance' between the unstable manifold of the DDE and the unstable manifold of the approximating ODE. 
We subsequently can determine this distance (for a concrete instance of a DDE and a fixed discretization index $n \in \mathbb{N}$) by means of the following computer-assisted strategy:

\begin{enumerate}
\item 
Numerically construct an `approximate zero' of the map $\zerops: \mathcal{S}_d \to \mathbb{S}_d$, i.e. construct a sequence $\hat{x} \in \mathcal{S}$ such that $\norm{\zerops (\hat{x})}$ is small
\item 
Construct a positive number $r_{\text{\tt PSA}} > 0$ such that the map $\zerops$ has an exact zero $x_{\text{\tt PSA}}$ on the ball $B(\hat{x}, r_{\text{\tt PSA}})$. 
\item 
Construct a positive number $r_{\text{\tt DDE}} > 0$ such that the map $\zerodde$ has an exact zero $x_{\text{\tt DDE}}$ on the ball $B(\hat{x}, r_{\text{\tt DDE}})$. 
\item The triangle inequality now implies that
\begin{equation} \label{eq:triangle}
\norm{x_{\text{\tt PSA}} - x_{\text{\tt DDE}}} < r_{\text{\tt PSA}} + r_{\text{\tt DDE}}.
\end{equation}
\end{enumerate}

In practice, we construct the constants $r_{\text{\tt PSA}}$ (resp. $r_{\text{\tt DDE}}$) by constructing two Newton-like maps of the form \eqref{eq:newton_intro} and by then showing that the map is a contraction on the ball $B(\hat{x}, r_{\text{\tt PSA}})$ (resp. $B(\hat{x}, r_{\text{\tt PSA}})$). We point out that this involves constructing two Newton-like maps (one for the ODE and one for the DDE), but only constructing one initial guess $\hat{x}$, since we will use the same initial guess for both maps. 

In Section \ref{sec:CAP}, we will perform the above strategy explicitly for Wright's equation
\begin{equation} \label{eq:wright}
\dot{x}(t) = - \alpha x(t-1)(1+x(t))
\end{equation}
with parameter $\alpha  = 2$ and with discretisation index $n  = \placen$. This yields the following result: 

\begin{theorem} \label{thm:distance_intro}
Consider Wright's equation \eqref{eq:wright} with $\alpha = 2$. 
Then the map $\zerodde$ (cf. \eqref{eq:zero_dde_intro}) has a zero $x_{\text{\tt DDE}}$ and the map $\zerops$ (cf. \eqref{eq:zero_ode_intro}) has zero $x_{\text{\tt PSA}}$ that satisfy 
\[
\norm{p_{\text{\tt DDE}} - p_{\text{\tt PSA}} } \leq 1.956812219647396 \times 10^{-9}. 
\]
\end{theorem}

\subsection{Outline of the paper}

Sections \ref{sec:ps_review} and \ref{sec:pm_ode_dde} review the concepts that the rest of the paper builds on, with Section \ref{sec:ps_review} reviewing the pseudospectral method for DDE and Section \ref{sec:pm_ode_dde} reviewing the parametrization method for both ODE and DDE. Section \ref{sec:parametrization ps} then discusses the parametrization method in the specific case of the pseudospectral ODE, leading to (a more formal formulation of) Theorem \ref{thm:thm_intro}. Section \ref{sec:CAP} discusses the computer-assisted approach to determining the distance between the unstable manifold of a DDE and the unstable manifold of its pseudospectral approximation, by means of the example of Wright's equation.

\section{Pseudospectral approximation of DDE}
\label{sec:ps_review}
In this section, we review the pseudospectral method as introduced in \cite{breda05, breda16}; parts of this section are based upon the exposition in \cite{dW21}. For notational simplicity, we restrict our attention to scalar-valued DDE, which we write as 
\begin{subequations}
\begin{equation} \label{eq:DDE}
\dot{x}(t) = Lx_t + G(x_t). 
\end{equation}
Here $L: C([-\tau, 0], \mathbb{R}) \to \mathbb{R}$ is a bounded linear operator and $G: C([-\tau, 0], \mathbb{R}) \to \mathbb{R}$ is a functional that satisfies $G(0) = G'(0) = 0$. The \emph{history segment} $x_t \in C([-\tau, 0], \mathbb{R})$ is defined as $x_t(\theta)  =x(t+\theta)$; hence, it satisfies the transport equation
\begin{equation} \label{eq:transport}
\partial_t x(t+\theta) = \partial_\theta x(t+\theta).
\end{equation}
\end{subequations}
The main idea of pseudospectral approximation the dynamics of \eqref{eq:DDE}--\eqref{eq:transport} with an ODE by approximating elements of the function space $C\left([-1, 0], \mathbb{R}\right)$ with polynomials. 

To do so, we fix a discretization index $n \in \mathbb{N}$ and choose $n+1$ meshpoints
\[ -1 \leq \theta_n < \ldots < \theta_0 = 0 \]
on the interval $[-1, 0]$. Various choices of this mesh are possible; through the rest of this article we use (a rescaled version of) the Chebyshev extremal nodes:
\begin{equation} \label{eq:chebyshev}
\theta_j = \frac{1}{2} \left( \cos \left(\frac{j \pi}{n} \right) -1 \right), \qquad 0 \leq j \leq n.
\end{equation}
Next, we consider the \emph{Lagrange polynomials}
\[ \ell_j(\theta) = \Pi_{\substack{0 \leq m \leq n \\ m \neq j}} \frac{\theta - \theta_m}{\theta_j - \theta_m}, \qquad 0 \leq j \leq n \]
that have the property that
\begin{subequations}
\begin{equation}
\begin{aligned}
\ell_j(\theta_i) = \delta_{ij} = \begin{cases}
0 \qquad \mbox{if } i \neq j; \\
1 \qquad \mbox{if } i = j.
\end{cases}
\end{aligned}
\end{equation}
\end{subequations}
We next approximate the history segment $x_t$ as 
\begin{equation*} 
x(t+\theta) \thicksim \sum_{j = 0}^n y_j(t) \ell_j(\theta),
\end{equation*}
so that in the approximation we seperate the time variable $t$ from the bookkeeping variable $\theta$. 
We require the analogue 
of \eqref{eq:DDE} to hold for $\theta_0 = 0$: 
\begin{subequations}
\begin{equation} \label{eq:def ode 1}
\dot{y}_0(t) = L \left(\sum_{j = 0}^n y_j(t) \ell_j\right) + G \left(\sum_{j = 0}^n y_j(t) \ell_j \right)
\end{equation}
and the analogue of \eqref{eq:transport} to hold for $\theta_1, \ldots, \theta_n$: 
\begin{equation} \label{eq:def ode 2}
 \dot{y}_k(t) = \sum_{j = 0}^n y_j(t) \ell_j'(\theta_k), \qquad k = 1, \ldots, n.
\end{equation}
\end{subequations}
The equations \eqref{eq:def ode 1}--\eqref{eq:def ode 2} together define a $n+1$-dimensional ODE, which we call the pseudospectral approximation of the DDE \eqref{eq:DDE}. 

In the following, we slightly simplify the expressions for the system \eqref{eq:def ode 1}--\eqref{eq:def ode 2}. It holds that 
\[ \sum_{j = 0}^n \ell_j(\theta) \equiv 1\]
since the function on the left hand side is a polynomial of degree $n$ that is equal to $1$ at $n+1$ meshpoints $\theta_0, \ldots, \theta_n$. From there it follows that 
\[ \ell_0'(\theta) = - \sum_{j = 1}^n \ell_j'(\theta), \qquad \theta \in [-1, 0] \]
and with this we can rewrite \eqref{eq:def ode 2} as
\[ \dot{y}_k(t) = - y_0(t) \left(\sum_{j = 1}^n \ell_j'(\theta_k) \right) + \sum_{j =1 }^n \ell_j'(\theta_k) y_j(t), \qquad k = 1, \ldots, n. \]
This system of equations we then more succinctly write as 
\begin{equation} \label{eq:ps_short_1}
\dot{y}(t) = -y_0 (t) D \textbf{1} - D y(t)
\end{equation}
where 
\begin{equation*}
y(t) = \begin{pmatrix}
y_1(t) \\ \vdots \\ y_n(t)
\end{pmatrix} \in \mathbb{R}^n
\end{equation*}
and 
\begin{equation} \label{eq:def D 1}
D = \begin{pmatrix}
\ell_1'(\theta_1) & \ldots & \ell_n'(\theta_1) \\
\vdots & \ddots & \vdots  \\
\ell_1'(\theta_n) & \ldots & \ell_n'(\theta_n)
\end{pmatrix}, \qquad \textbf{1} = \begin{pmatrix}
1 \\ 
\vdots \\
1
\end{pmatrix} \in \mathbb{R}^n
\end{equation}
(both the matrix $D$ and the vector $\textbf{1}$ depend on the discretization index $n$, but we suppress this in notation). 
Next we define the operator 
\begin{equation} \label{eq:def P}
\begin{aligned}
P: \mathbb{R} \times \mathbb{R}^n \to C\left([-1, 0], \mathbb{R}\right) \\
P(y_0, y) = y_0 \ell_0 + \sum_{j = 1}^n y_j \ell_j
\end{aligned}
\end{equation}
so that we can rewrite \eqref{eq:def ode 1} as 
\begin{equation} \label{eq:ps_short_2}
\dot{y}_0(t) = LP(y_0, y) + G(P(y_0, y)). 
\end{equation}
So in this notation, the pseudospectral approximation is given by the equations \eqref{eq:ps_short_1} and \eqref{eq:ps_short_2}, as we summarize in the following definition: 

\begin{definition}
The pseudospectral approximation of the DDE 
\[ \dot{x}(t) = Lx_t + G(x_t) \]
is given by the $n+1$ dimensional system of ODE 
\begin{equation}\label{eq:ps ode}
\begin{pmatrix}
\dot{y}_0 \\
\dot{y}
\end{pmatrix} = \begin{pmatrix}
LP(y_0, y) + G(P(y_0, y)) \\
- y_0 D \textbf{1} + D y
\end{pmatrix}
\end{equation}
where $y_0 \in \mathbb{R}$, $y = (y_1, \ldots, y_n)^T \in \mathbb{R}^n$, $D \in \mathbb{R}^{n \times n}$ and $\textbf{1} \in \mathbb{R}^n$ are defined in \eqref{eq:def D 1} and the operator $P$ is defined in \eqref{eq:def P}. 
\end{definition}

\begin{example}[Wright's equation] \label{ex:wright}
Wright's equation
\begin{equation} \label{eq:wright}
\dot{x}(t) = - \alpha x(t-1)(1+x(t)),
\end{equation}
with $x(t) \in \mathbb{R}$ and parameter $\alpha \in \mathbb{R}$, is of the form \eqref{eq:DDE} with $L, G: C([-1, 0], \mathbb{R}) \to \mathbb{R}$ given by 
\begin{equation} \label{eq:L G wright}
Lx_t = - \alpha x(t-1)  \qquad \mbox{and} \qquad G(x_t) = - \alpha x(t-1) x(t). 
\end{equation}
For the Chebyshev nodes \eqref{eq:chebyshev}, the node $\theta_n$ is placed at $\theta_n = -1$; therefore, for any vector $(y_0, y) \in \mathbb{R} \times \mathbb{R}^n$ the interpolating polynomial $P(y_0, y)$ takes the value $y_n$ at $\theta_n = -1$. Hence it holds that 
\[ LP(y_0, y) = - \alpha y^{(n)}, \qquad G(P(y_0, y)) = - \alpha y^{(n)} y_0 \]
and the pseudospectral approximation to Wright's equation is given by 
\[ 
\begin{pmatrix}
\dot{y}_0 \\ \dot{y}
\end{pmatrix}
= \begin{pmatrix}
-\alpha y^{(n)} (1+y_0) \\
D y - y_0 D \textbf{1}
\end{pmatrix}
\]
\end{example}

\section{Parametrization method for ODE and DDE} \label{sec:pm_ode_dde}
In this section, we review the parametrization methods for ODE and DDE. 
Both for ODE and DDE, we consider an equilibrium of some dynamical problem which has $d$ unstable eigenvalues. The idea of the parametrization method is that we view its local unstable manifold as a graph of a function, and then recursively determine the Taylor coefficients of this function. We first review the necessary sequence spaces in Subsection \ref{sec:sequences}; we then review the parametrization method for ODE in Subsection \ref{sec:parametrization ode} and the parametrization method for DDE in Subsection \ref{sec:parametrization dde}. 

\subsection{Sequence spaces and analytic functions} \label{sec:sequences}
If $g: \dd \to \cc$ is an analytic function on the disk 
\[
\dd = \setof*{z \in \cc : \abs*{z} < 1},
\]
then it admits a unique series expansion of the form 
\[
g(z) = \sum_{\beta = 0}^\infty a_\beta z^\beta. 
\]
We can uniquely identify $g$ with the coefficient sequence $\setof*{a_\beta}_{\beta \in \nn}$, which we view as an element of the space of complex scalar sequences
\[
\cS := \setof{\setof{u_\beta}_{\beta \in \nn} \mid u_\beta \in \cc}.
\]
Similarly, an analytic function $g: \dd^d \to \cc$  on the polydisk 
\[
\dd^d := \setof*{z \in \cc^d : \abs*{z_j} < 1, \ 1 \leq j \leq d}
\]
has a $d$-variate Taylor expansion and therefore is uniquely identified with a $d$-infinite sequence in the space
\[
  \cS_d := \underbrace{\cS \otimes \dots \otimes \cS}_{d \ \text{copies}} = \setof*{\setof*{u_\beta}_{\beta \in \nn^d} \mid u_\beta \in \cc}. 
\]
Finally, if $g: \dd^d \to \cc^n$ is a \emph{vector-valued} analytic function, 
then it has a Taylor expansion of the form 
\[
    g(z) = \sum_{\beta \in \nn^d} a_\beta z^\beta, \qquad a_\beta \in \cc^n.
\]
which we an uniquely identify with a sequence in the space
\[
\cS_d^n = \underbrace{\cS_d \times \dots \times \cS_d}_{n \ \text{copies}} = \setof*{(u_1, \dotsc, u_n) : u_j \in \cS_d, \ 1 \leq j \leq n}.
\]
The space $\cS_d^n$ is isomorphic (as a vector space) to the space of $d$-infinite complex \emph{vector sequences}
\[
\cS_d^n \cong \setof*{\setof*{u_\beta}_{\beta \in \nn^d} \mid u_\beta \in \cc^n}
\]
which is a convenient space to represent $g$.  

In the rest of Section \ref{sec:pm_ode_dde} and in Section \ref{sec:parametrization ps}, we discuss Taylor sequences formally,
without regard to convergence; in this case, we focus only on  
the vector space structure of $\cS^n_d$. In Section \ref{sec:CAP}, we will consider convergence of the Taylor series of scalar-valued functions $g: \mathbb{D}^d \to \mathbb{C}$; to do so, 
we will equip $\cS_d$ with a norm to measures decay of the coefficients. 
For a scalar sequence $u = \setof*{u_\beta} \in \cS_d$ we will use the $\ell^1$ norm defined by 
\begin{equation}
\label{eq:ell1_norm}
\norm{u}_{\ell^1} := \sum_{\beta \in \nn^d} \abs{u_\beta}.
\end{equation}
In particular, the set
\[
    \ell^1 := \setof*{u \in \cS_d : \norm{u}_{\ell^1} < \infty} \subset \cS_d
\]
is a Banach space with respect to this norm.

\medskip

When computing the Taylor coefficients of function numerically, we can of course only hope to compute finitely many coefficients. To account for this in notation, 
we define for fixed $N \in \mathbb{N}$ the projection operator $\pi_N$ on $X = \cS_d^n$ as
\[
(\pi_N (u))_\beta = 
\begin{cases}
u_\beta &\qquad \mbox{if } \abs{\beta} \leq N \\
0 &\qquad \mbox{if } \abs{\beta} > N.
\end{cases}
\]
and define the projection operator $\pi_\infty$ as
\[
(\pi_\infty (u))_\beta = 
\begin{cases}
0 &\qquad \mbox{if } \abs{\beta} \leq N \\
u_\beta &\qquad \mbox{if } \abs{\beta} > N
\end{cases}
\]
so that $\pi_N + \pi_\infty$ is equal to the identity operator on $X$. We let $X_N$ denote the image of the projection $\pi_N$, i.e.
\[
X_N := \{ u \in \mathcal{S}_d^n \mid u_\beta = 0 \quad \mbox{for} \quad \abs{\beta} \geq N \}; 
\]
and we let $X_\infty$ denote the image of the projection $\pi_\infty$ so that $X = X_N \oplus X_\infty$. If now $u  = \{u_d \}_{\abs{d} \leq N}$ is a finite sequence (for example, a finite sequence that we have computed numerically), then we can naturally identify it with a sequence in $X_N$, by `extending' the definition of $u$ via $u_\beta = 0$ for $\abs{\beta} > N$. 
Hence in the following, we will not always distinguish between a finite sequence and a sequence in $X_N$. 

For an operator $T: X \to X$, we denote by 
\[
\norm{T}_{\ell^1} := \sup_{\substack{u \in X \\ \norm*{u}_{\ell^{1}} = 1}} \{ \norm*{T u}_{\ell^{1}} \}
\]
the operator norm of $T$. Hence the notation $\norm{.}_{\ell^{1}}$ can either refer to the $\ell^{1}$ norm on $X$ or the induced operator norm but the distinction will be clear from the context.

\subsection{Parametrization method for ODE} \label{sec:parametrization ode}
We next review the parametrization method for ODE as introduced in \cite{Cabre03, Cabre05}. We consider the ODE 
\begin{equation} \label{eq:ode}
\dot{x}(t) = \linode x(t) + \nonlinode(x(t))
\end{equation}
with $x(t) \in \mathbb{R}^n$ and $\linode \in \mathbb{R}^{n \times n}$ a matrix. We moreover assume that the smooth function $\nonlinode: \mathbb{R}^n \to \mathbb{R}^n$ satisfies $\nonlinode(0) = 0$, $\nonlinode'(0) = 0$, so that the ODE \eqref{eq:ode} has an equilibrium $x \equiv 0$. The  linearization of \eqref{eq:ode} around this equilibrium is given by
\[ \dot{x}(t) = \linode x(t); \]
if the matrix $\linode$ has $d$ eigenvalues 
\[ \lambda_1, \ldots, \lambda_d \]
in the right half of the complex plane, then the equilibrium $x \equiv 0$ of \eqref{eq:ode} has a $d$-dimensional unstable manifold. The local unstable manifold (i.e. the unstable manifold in a neighbourhoud of the origin) is a graph of a map
\[ \pode: \dd^d \to \mathbb{R}^n;\]
and we call the map $\pode$ a \emph{parametrization} of the local unstable manifold. Although the local unstable manifold is unique, its parametrization is not. 
The idea of the parametrization method is to recursively compute the Taylor coefficients of a particularly convenient parametrization $\pode$. 

To explain the parametrization method for ODE in the simplest context, we make the additional assumption that the eigenvalues $\lambda_1, \ldots, \lambda_d$ are non-resonant, i.e. 
\[ \beta_1 \lambda_1 + \ldots + \beta_d \lambda_d \not \in \sigma(\linode) \]
for all $\beta \in \nn^d$ with $\abs*{\beta} \geq 2$. In this case, a particularly convenient parametrization $\pode$ is the parametrization that sends solutions of the ODE
\begin{equation} \label{eq:flow mfd}
\dot{y}(t) = \Lambda y(t) + h(y(t))
\end{equation}
to solutions of the ODE \eqref{eq:ode}. Here $\Lambda$ is a diagonal matrix given by
\[ \Lambda = \begin{pmatrix}
\lambda_1 & 0 & \ldots & 0 \\
0 & \lambda_2 & \ldots & 0 \\
0 & 0 & \ldots & \lambda_d
\end{pmatrix} \]
and the function $h$ is a so-called ``flat'' function, i.e. it has the property that 
\[ D^{(\alpha)}h(0) = 0 \]
for all $\alpha \in \mathbb{N}^d$. The parametrization $\pode$ sends solutions of \eqref{eq:flow mfd} to solutions of \eqref{eq:ode} if and only if 
\begin{equation} \label{eq:conjugacy}
\linode \pode (\sigma) + \nonlinode (\pode(\sigma)) = 
\pode '(\sigma) \Lambda \sigma + \pode '(\sigma) h(\sigma).
\end{equation}
If we substitute the Ansatz 
\[ \pode (\sigma) = \sum_{\substack{\beta \in \mathbb{N}^d \\ \left| \beta \right| \leq k}} \pode_\beta \sigma^\beta + \mathcal{O}(\sigma^{k+1}); \]
into \eqref{eq:conjugacy}, we obtain
\begin{equation} \label{eq:conjugacy ode full}
\sum_{\substack{\beta \in \mathbb{N}^d \\ \left| \beta \right| \leq k}} \linode \pode_\beta \sigma^\beta + \nonlinode\left(\sum_{\substack{\beta \in \mathbb{N}^d \\ \left| \beta \right| \leq k}} \pode_\beta \sigma^\beta \right) + \mathcal{O}(\sigma^{k+1}) = \sum_{\substack{\beta \in \mathbb{N}^d \\ \left| \beta \right| \leq k}} \dotp{\lambda, \beta} \pode_\beta \sigma^\beta + \mathcal{O}(\sigma^{k+1}),
\end{equation}
where $\dotp{\lambda, \beta} = \lambda_1 \beta_1 + \ldots \lambda_d \beta_d$ (in particular, the flat term $h$ `disappears' from the equation if we consider the equation up to finite order in its Taylor expansion). Since the map $\nonlinode$ is smooth  and $\nonlinode(0)$ and $D \nonlinode (0) = 0$, we can write  
\begin{equation} \label{eq:g taylor}
\nonlinode\left(\sum_{\substack{\beta \in \mathbb{N}^d \\ \left| \beta \right| \leq k}} \pode_\beta \sigma^\beta \right) = \sum_{\substack{\beta \in \mathbb{N}^d \\ \left| \beta \right| \leq k}} \tilde{\pode}_\beta \sigma^\beta + \mathcal{O}(\sigma^{k+1}),
\end{equation}
where every coefficient $\tilde{\pode}_\beta$ can be expressed in terms of the coefficients $\pode_{\beta'}$ with $\left| \beta' \right| < \left| \beta \right|$. If we now substitute \eqref{eq:g taylor} in \eqref{eq:conjugacy ode full} and equate terms of order $\beta$ for $|\beta| \geq 2$, we obtain
\begin{subequations}
\begin{equation} \label{eq:hom eqn ode}
\left(\dotp{\lambda, \beta}I - \linode \right)\pode_\beta = \tilde{\pode}_\beta, \qquad |\beta| \geq 2. 
\end{equation}
We supplement this with the conditions
\begin{equation} \label{eq:order zero one ode}
\pode_{0} = 0 \qquad \mbox{and} \qquad \pode_{\beta} = \xi_\beta, \, \mbox{for } | \beta| = 1,
\end{equation}
\end{subequations}
where for each vector $\beta$ with $|\beta| = 1$, the vector $\xi_\beta$ is an eigenvector
$\linode$ with eigenvalue $\lambda_\beta$. Hence the condition \eqref{eq:order zero one ode} ensures that the graph of $\pode$ is attached to the origin and tangent to the unstable eigenspace of $\linode$. Equations \eqref{eq:order zero one ode}--\eqref{eq:hom eqn ode} give us a method to compute the coefficients $\pode_\beta, \ \beta \in \mathbb{N}^d$ recursively: if for a given $\beta$, we have computed the coefficients $\{ \pode_{\beta'} \}_{\left| \beta' \right| < \left| \beta \right|}$, we can compute $\tilde{\pode}_\beta$ and then solve for $\pode_\beta$ from \eqref{eq:hom eqn ode}.

\medskip

In the above discussion, we first invoked an abstract existence result for unstable manifolds, and then argued that the Taylor coefficients of the map $\pode$ can be computed recursively. But in specific cases, it is sometimes also possible to argue the other way around. If for a concrete ODE, there exists a sequence of coefficients $\{\pode_\beta\}_{\beta \in \mathbb{N}^d}$ that satisfy \eqref{eq:order zero one ode} and \eqref{eq:hom eqn ode} and additionally the sequence
\begin{equation} \label{eq:analytic ode}
\sigma \mapsto \sum_{\beta \in \mathbb{N}^d} \pode_\beta \sigma^\beta
\end{equation}
is real analytic on some open ball $B(0, r) \subseteq \mathbb{C}^d$, then locally the unstable manifold is given by the graph of \eqref{eq:analytic ode}. As a consequence, the local unstable manifold is then also analytic. 
In practice, we can check that the function \eqref{eq:analytic ode} is analytic by checking that 
\[ \sum_{\beta \in \mathbb{N}^d} \left| \pode_\beta \right| < \infty. \]

\medskip

We summarize the above discussion in the next hypothesis and proposition, where we denote by $\mathcal{S}_d^n$ the sequence space introduced in Subsection \ref{sec:sequences}.

\begin{hyp} \label{hyp:ODE}
We make the following assumptions on the ODE \eqref{eq:ode}: 
\begin{enumerate}
\item The map $\nonlinode: \mathbb{R}^n \to \mathbb{R}^n$ is smooth (i.e. it is $C^k$ for every $k \in \mathbb{N}$) and satisfies $\nonlinode (0) = 0, \ D\nonlinode (0) = 0$. 
\item The matrix $\linode$ has $d$ roots $\lambda_1, \ldots, \lambda_d$ in the right half of the complex plane; and these roots satisfy the non-resonance condition 
\[ \beta_1 \lambda_1 + \ldots + \beta_d \lambda_d \not \in \sigma(\linode) \]
for all multi-indices $\beta = (\beta_1, \ldots, \beta_d)$ with $\left| \beta \right| := |\beta_1| + \ldots + |\beta_d| \geq 2$.
\end{enumerate}
\end{hyp}

\begin{proposition}[cf. \cite{Cabre03}] \label{prop:pm ode}
Consider the ODE \eqref{eq:ode} satisfying Hypothesis \ref{hyp:ODE}. For every multi-index $\beta \in \mathbb{N}^d$ with $\abs{\beta} = 1$, let $\xi_\beta$ be an eigenvector of $\linode$ with eigenvalue $\lambda_\beta$. Denote by 
\[ \begin{aligned}
\hoode: \mathcal{S}_d^n &\to \mathcal{S}_d^n \\
\end{aligned}
 \]
the map with the property that 
\[ \nonlinode \left(\sum_{\substack{\beta \in \mathbb{N}^d \\ | \beta | \leq k }} x_\beta \sigma^\beta \right) = \sum_{\substack{\beta \in \mathbb{N}^d \\ | \beta | \leq k }} \hoode (x)_{\beta} \sigma^\beta + \mathcal{O}(\sigma^{k+1}) \]
for every sequence $x = (x_\beta)_{\beta \in \mathbb{N}^d}$. Moreover, define the map $\zeroode: \mathcal{S}_d^n \to \mathcal{S}_d^n$ as 
\begin{align*}
 \zeroode(x)_\beta = \begin{cases}
0 \qquad &\mbox{if } \beta = 0 \\
x_\beta - \xi_\beta \qquad &\mbox{if } \abs{\beta} = 1, \\
\left(\dotp{\lambda, \beta}I - A \right)x_\beta - H(x)_\beta, \qquad &\mbox{if } | \beta | \geq 2. 
\end{cases}
\end{align*}
If the sequence $\pode = \{ \pode_\beta \}_{\beta \in \mathbb{N}^d}$ has the properties 
\[ \zeroode( \pode ) = 0 \qquad \mbox{ and } \qquad \sum_{\beta \in \mathbb{N}^d} \left| \pode_\beta \right| < \infty  \]
then the graph of the function 
\begin{align*}
\pode: \mathbb{C}^d \supseteq B(0, 1) \to \mathbb{R}^n \\
\pode(\sigma) = \sum_{\beta \in \mathbb{N}^d} \pode_\beta \sigma^\beta
\end{align*}
is a local unstable manifold for the ODE \eqref{eq:ode}. 
\end{proposition}




Before continuing, we briefly remark about zero finding problems and computer assisted proof in general. Specifically, we conider the question of solving a zero finding problem of the form $\cF(\cP) = 0$ in the case that $\cF : X \to Y$ is a map between (possibly infinite dimensional) Banach spaces. There are a number of strategies for this which are based on computer assisted proof. In this work we employ two ``constructive'' versions of the Contraction Mapping Theorem based on results initially found in \cite{MR1639986,MR2338393}. The explicit statements for each theorem can be found in Appendix \ref{appendix:contraction_mapping}. 

Both cases include obtaining a highly accurate numerical approximation as a key ingredient. Specifically, using classical numerical techniques we obtain $\hat{x} \in X$ satisfying $\cF(\hat{x}) \approx 0$. It should also be true that $\hat{x}$ can be represented finitely in some way which we can store exactly in a computer. With this in hand we attempt to identify a positive real number $r$ for which we can prove that $\cF \big|_{B_r(\hat{x})}$ is a contraction. The methods for finding $r$ and for verifying that this restriction is a contraction vary from problem to problem and are best demonstrated in concrete examples which we do in Section \ref{sec:CAP}.

\subsection{Parametrization method for DDE} \label{sec:parametrization dde}

We next review the parametrization method for DDE. For notational simplicity, we restrict again our attention to scalar valued DDE, i.e. DDE of the form \eqref{eq:DDE} where 
$L: C \left([-1, 0], \mathbb{R}\right) \to \mathbb{R}$ is linear and $G: C \left([-1, 0], \mathbb{R}\right) \to \mathbb{R}$ is smooth. Under the additional assumption $G(0) = G'(0) = 0$, the DDE \eqref{eq:DDE} has an equilibrium $x \equiv 0$,
and the linearization of \eqref{eq:DDE} around this equilibrium is given by
\begin{equation} \label{eq:linear DDE}
\dot{x}(t) = L x_t, \qquad t \geq 0. 
\end{equation}
The linear DDE \eqref{eq:linear DDE} has a solution of the form $\lambda \mapsto e^{\lambda t}, \ \lambda \in \mathbb{C}$ if and only if $\lambda$ is a root of the equation 
\begin{equation} \label{eq:ce}
\lambda - L \varepsilon_\lambda = 0, 
\end{equation}
where the function $\varepsilon_\lambda \in C\left([-1, 0], \mathbb{R}\right)$ is defined as 
\begin{equation} \label{eq:varepsilon}
\varepsilon_\lambda(\theta) = e^{\lambda \theta} \qquad \theta \in [-1, 0].
\end{equation}
Equation \eqref{eq:ce} is called the \emph{characteristic equation} of the DDE \eqref{eq:linear DDE}; the function 
\begin{equation}
\Delta: \mathbb{C} \to \mathbb{C}, \qquad \Delta(\lambda) = \lambda - L \varepsilon_\lambda
\end{equation}
is called the \emph{characteristic function}. 

If the characteristic equation \eqref{eq:ce} has $d$ roots
\[ \lambda_1, \ldots, \lambda_d \]
in the right half of the complex plane, then the equilibrium $x \equiv 0$ of \eqref{eq:DDE} has a $d$-dimensional unstable manifold \cite{Diekmann95}. 
Locally, this unstable manifold is the graph of a function
\begin{equation}
\pdde: \mathbb{C}^d \supseteq B(0, 1) \to C \left([-1, 0], \mathbb{R}\right). 
\end{equation}
Since the function $\pdde$ takes values in a function space, the coefficients $P_\beta$ in the Taylor expansion
\[ \pdde(\sigma) = \sum_{\substack{\beta \in \mathbb{N}^d \\ 0 \leq \left|\beta \right| \leq k}} \pdde_\beta \sigma^\beta + \mathcal{O}(\sigma^{k+1})\]
are elements of $C \left([-1, 0], \mathbb{R}\right) \to \mathbb{R}$, i.e. the coefficients are in turn functions. In \cite{Groothedde17}, the authors showed that each function $P_\beta$ is of the form 
\[ \pdde_\beta = p_\beta \varepsilon_{\left \langle \lambda, \beta \right \rangle}, \]
where  $\lambda = (\lambda_1, \ldots, \lambda_d) \in \mathbb{C}^d$, the exponential functions $\varepsilon_{\dotp{\lambda, \beta}}$ are defined via \eqref{eq:varepsilon}, and the coefficients $p_\beta \in \mathbb{R}$ are now \emph{scalars}. 

Under the additional \emph{non-resonance condition}
\begin{equation*} 
\Delta(\left \langle \lambda, \beta \right \rangle) \neq 0 \qquad \mbox{for all } \beta = (\beta_1, \ldots, \beta_d) \in \mathbb{N}^d \quad \mbox{with} \quad |\beta| > 2,  
\end{equation*}
the authors of \cite{Groothedde17} show that the coefficients $p_\beta \in \mathbb{R}$ satisfy
\begin{equation} \label{eq:conjugacy dde}
\sum_{\substack{\beta \in \mathbb{N}^d \\ 0 \leq \left|\beta \right| \leq k}} \dotp{\lambda, \beta} p_\beta \sigma^\beta = \sum_{\substack{\beta \in \mathbb{N}^d \\ 0 \leq \left|\beta \right| \leq k}} p_\beta L  \varepsilon_{\dotp{\lambda, \beta}} \sigma^\beta + G \left(\sum_{\substack{\beta \in \mathbb{N}^d \\ 0 \leq \left|\beta \right| \leq k}} p_\beta \varepsilon_{\dotp{\lambda, \beta}}\sigma^\beta \right) +\mathcal{O}(\sigma^{k+1})
\end{equation}
for each $k \in \mathbb{N}$. 
Since the smooth function $G$ satisfies $G(0) = 0, \ G'(0) = 0$, we can write
\begin{equation} \label{eq:G composition}
G \left(\sum_{\substack{\beta \in \mathbb{N}^d \\ 0 \leq \left|\beta \right| \leq k}} p_\beta \varepsilon_{\dotp{\lambda, \beta}}\sigma^\beta\right) = \sum_{\substack{\beta \in \mathbb{N}^d\\ 0 \leq \left|\beta \right| \leq k}} \tilde{p}_\beta \sigma^\beta + \mathcal{O}(\sigma^{k+1}).
\end{equation}
where each of the coefficient $\tilde{p}_\beta \in \mathbb{C}$ can be expressed in terms of $p_{\beta'}$ with $\left| \beta' \right| < \left| \beta \right|$. 
We then rewrite \eqref{eq:conjugacy dde} as
\[ \sum_{\beta \in \mathbb{N}^d} \dotp{\lambda, \beta} p_\beta \sigma^\beta - \sum_{\beta \in \mathbb{N}^d} L p_\beta \varepsilon_{\dotp{\lambda, \beta}} \sigma^\beta = \sum_{\beta \in \mathbb{N}^d} \tilde{p}_\beta \sigma^\beta + \mathcal{O}(\sigma^{k+1}). \]
If we equate terms of order $\beta$ for $\abs{\beta} \geq 2$, we obtain that
\[ \dotp{\lambda, \beta} p_\beta - L p_\beta \varepsilon_{\dotp{\lambda, \beta}} = \tilde{p}_\beta \]
or equivalently that
\begin{subequations}
\begin{equation} \label{eq:scalar dde}
\Delta(\dotp{\lambda, \beta}) p_\beta = \tilde{p}_\beta
\end{equation}
for $\abs{\beta} \geq 2$. 
We supplement \eqref{eq:scalar dde} with the constraints
\begin{equation} \label{eq:order zero one dde}
\begin{aligned}
\begin{cases}
p_\beta = 0 \qquad &\mbox{for } \beta = 0; \\
p_\beta = \xi_\beta \qquad &\mbox{for } | \beta | = 1;
\end{cases}
\end{aligned}
\end{equation}
\end{subequations}
where for each $\abs{\beta} = 1$, $\xi_\beta \in \mathbb{R}$ is a (fixed) constants that can be chosen freely. Equations \eqref{eq:scalar dde}--\eqref{eq:order zero one dde} now give us a method to compute the coefficients $p_\beta, \ \beta \in \mathbb{N}^d$ recursively: if for a given $\beta$, we have computed the coefficients $\{ p_{\beta'} \}_{\left| \beta' \right| < \left| \beta \right|}$, we can compute $\tilde{p}_\beta$ and then solve for $p_\beta$ from \eqref{eq:scalar dde}. Moreover, if the sequence $(p_\beta)_{\beta \in \mathbb{N}^d}$ satisfies \eqref{eq:scalar dde}--\eqref{eq:order zero one dde} and addtionally satisfies
\[ \sum_{\beta \in \mathbb{N}^2} \left| p_\beta \right| < \infty, \]
then the graph of the function 
\begin{equation} \label{eq:analytic}
\sigma \mapsto \sum_{\beta \in \mathbb{N}^d} p_\beta \varepsilon_{\dotp{\lambda, \beta}} \sigma^\beta
\end{equation}
is the local unstable manifold of the DDE \eqref{eq:DDE}; cf. \cite{Henot22}. 

\medskip

We summarize the above discussion in the next hypothesis and proposition.

\begin{hyp} \label{hyp:DDE}
We make the following assumptions on the DDE \eqref{eq:DDE}: 
\begin{enumerate}
\item The map $L: C \left([-1, 0], \mathbb{R}\right) \to \mathbb{R}$ is a bounded linear operator;
\item The map $G: C \left([-1, 0], \mathbb{R}\right) \to \mathbb{R}$ is smooth (i.e. it is $C^k$ for every $k \in \mathbb{N}$) and satisfies $G(0) = 0, \ G'(0) = 0$. 
\item The characteristic equation \eqref{eq:ce} has $d$ roots $\lambda_1, \ldots, \lambda_d$ in the right half of the complex plane; and these roots satisfy the non-resonance condition 
\begin{equation} \label{eq:non-resonance}
\Delta(\left \langle \lambda, \beta \right \rangle) \neq 0 \qquad \mbox{for all } \beta = (\beta_1, \ldots, \beta_d) \in \mathbb{N}^d \quad \mbox{with} \quad |\beta| > 2. 
\end{equation}
\end{enumerate}
\end{hyp}

\begin{proposition}[cf. \cite{Henot22}] \label{prop:pm dde}
Consider the DDE \eqref{eq:DDE} satisfying Hypothesis \ref{hyp:DDE}. Denote by 
\[ 
\hodde: \mathcal{S}_d \to \mathcal{S}_d 
\]
the map with the property that 
\begin{equation} \label{eq:H dde}
G\left(\sum_{\substack{\beta \in \mathbb{N}^d \\ | \beta | \leq k }} x_\beta \varepsilon_{\dotp{\lambda, \beta}} \sigma^\beta \right) = \sum_{\substack{\beta \in \mathbb{N}^d \\ | \beta | \leq k }} \hodde(x)_{\beta} \sigma^\beta + \mathcal{O}(\sigma^{k+1}) 
\end{equation}
for every sequence $x = (x_\beta)_{\beta \in \mathbb{N}^d}$. For each $\beta$ with $|\beta| = 1$, fix a scalar $\xi_\beta \in \mathbb{R}$ and define the map 
\[  \zerodde: \mathcal{S}_d \to \mathcal{S}_d \]
as 
\begin{align*}
\zerodde (x)_\beta = \begin{cases}
0 \qquad &\mbox{for } \beta = 0, \\
x_\beta - \xi_\beta \qquad &\mbox{for } |\beta|  = 1, \\
\Delta(\dotp{\lambda, \beta}) x_\beta - H(x)_\beta, \qquad &\mbox{for } | \beta | \geq 2.
\end{cases}
\end{align*}
If the sequence $p = \{ p_\beta \}_{\beta \in \mathbb{N}^d}$ has the properties 
\[ \zerodde (p) = 0 \qquad \mbox{ and } \qquad \sum_{\beta \in \mathbb{N}^d} \left| p_\beta \right| < \infty,\]
then the graph of the function 
\begin{align*}
\pdde: \mathbb{C}^d \supseteq B(0, 1) \to C \left([-1, 0], \mathbb{R}\right) \\
\pdde (\sigma) = \sum_{\beta \in \mathbb{N}^d} p_\beta \varepsilon_{\dotp{\lambda, \beta}} \sigma^\beta
\end{align*}
is a local unstable manifold for the DDE \eqref{eq:DDE}. 
\end{proposition}

\noindent

Wright's equation
\begin{equation*} 
\dot{x}(t) = - \alpha x(t-1)(1+x(t)),
\end{equation*}
with $x(t) \in \mathbb{R}$ and parameter $\alpha \in \mathbb{R}$, has an equilibriam $x \equiv 0$; the characteristic function of the linearized equation is given by 
\begin{equation} \label{eq:ce wright}
\Delta(\lambda) = \lambda + \alpha e^{-\lambda}.
\end{equation}
For $\frac{\pi}{2} < \alpha < \frac{5 \pi}{2}$, the characteristic function $\Delta$ has exactly two zeroes in the right half of the complex plane (see, for example, \cite{Hayes50}, \cite[Appendix]{HaleVL93}, \cite[Chapter XI]{Diekmann95}, or one of the many other references that study the roots of the transcedental function \eqref{eq:ce wright}). We denote these eigenvalues by 
\[ \lambda_+, \lambda_- \]
where $\lambda_+$ has positive imaginary part and $\lambda_-$ has negative imaginary part. Since $\lambda_+, \ \lambda_-$ are the only zeroes of $\Delta$ in the right half of the complex plane, the non-resonance condition \eqref{eq:non-resonance} is automatically satisfied. Hence we can apply Proposition \ref{prop:pm dde} to Wright's equation: 

\begin{corollary} \label{cor:pm_wright_dde}
Consider Wright's equation
\begin{equation} \label{eq:wright_lemma}
\dot{x}(t) = - \alpha x(t-1)(1+x(t))
\end{equation}
with $x(t) \in \mathbb{R}$ and $\frac{\pi}{2} < \alpha < \frac{5 \pi}{2}$. Define the function $\Delta$ as in \eqref{eq:ce wright} and denote by $\lambda = (\lambda_+, \lambda_-)$
the two zeroes of the $\Delta$ in the right half of the complex plane.

For $\beta \in \mathbb{N}^2$ with $|\beta| = 1$, fix scalars $\xi_\beta \in \mathbb{R}$ and define the function
$F: \mathcal{S}_2 \to \mathcal{S}_2$ as 
\begin{equation} \label{eq:zero_map_wright_dde}
\begin{aligned}
F(x)_\beta = \begin{cases}
x_\beta \qquad &\mbox{if } \beta = 0, \\
x_\beta - \xi_\beta \qquad &\mbox{if } |\beta| = 1, \\
\Delta(\dotp{\lambda, \beta}) x_\beta + \alpha (x*\blowupscalardde(x)), \qquad &\mbox{if } | \beta | \geq 2,
\end{cases}
\end{aligned}
\end{equation}
where the operator $\blowupscalardde: \mathcal{S}_2 \to \mathcal{S}_2$ is defined via the relation 
\begin{equation} \label{eq:def R}
\blowupscalardde(x)_\beta = x_\beta e^{-\dotp{\lambda, \beta}}, \qquad \beta \in \mathbb{N}^2,
\end{equation} 
and the convolution operator $*: \mathcal{S}_2 \times \mathcal{S}_2 \to \mathcal{S}_2$ is defined via the relation 
\begin{equation} \label{eq:def convolution}
(x*y)_\beta = \sum_{\beta_1 + \beta_2 = \beta} x_{\beta_1} y_{\beta_2}, \qquad \beta \in \mathbb{N}^2. 
\end{equation}
Then, if $p$ is a sequence such that $F(p) = 0$ and $\sum_{\beta} | p_\beta| < \infty$, 
then the graph of the function
\[ P(\sigma) = \sum_{\beta \in \mathbb{N}^2} p_\beta \varepsilon_{\dotp{\lambda, \beta}} \sigma^\beta\]
is the local unstable manifold around the origin of \eqref{eq:wright_lemma}. 
\end{corollary}
\begin{proof}
Wright's equation \eqref{eq:wright_lemma} is of the form \eqref{eq:DDE} with the functionals $L$ and $G$ given in \eqref{eq:L G wright}. Given a sequence $p = (p_\beta)_{\beta \in \mathbb{N}^2}$, we (formally) compute that 
\begin{align*}
G\left(\sum_{\beta \in \mathbb{N}^2} p_\beta \varepsilon_{\dotp{\lambda, \beta}} \sigma^\beta \right) &= - \alpha \left(\sum_{\beta \in \mathbb{N}^2} p_\beta e^{-\dotp{\lambda, \beta}} \sigma^\beta \right) \left(\sum_{\beta \in \mathbb{N}^2} p_\beta \sigma^\beta \right) \\
&= - \alpha \sum_{\beta \in \mathbb{N}^2} \left(\sum_{\beta_1 + \beta_2 =  \beta} p_{\beta_1} e^{-\dotp{\lambda, \beta_1}} p_{\beta_2} \right) \sigma^\beta
\end{align*}
so that the function $\hodde: \mathcal{S}_2 \to \mathcal{S}_2$ in \eqref{eq:H dde} is in this case given by 
\begin{equation} \label{eq:H wright}
\hodde(p)_\beta = - \alpha \sum_{\beta_1 + \beta_2 =  \beta} p_{\beta_1} e^{-\dotp{\lambda, \beta_1}} p_{\beta_2}. 
\end{equation}
With the notation \eqref{eq:def R}, \eqref{eq:def convolution}, we can rewrite this as 
\[ \hodde(p) = - \alpha(p*\blowupscalardde(p)). \]
An application of Proposition \ref{prop:pm dde} now proves the claim. 
\end{proof}

\section{Parametrization method for pseudospectral ODE}

\label{sec:parametrization ps}
In this section, we apply the parametrization method for ODE in the specific case where the ODE is a pseudospectral approximation of a DDE.
To that end, we again consider the DDE 
\eqref{eq:DDE} with $L: C([-\tau, 0], \mathbb{R}) \to \mathbb{R}$ is a bounded linear operator and $G: C([-\tau, 0], \mathbb{R}) \to \mathbb{R}$ a smooth functional that satisfies $G(0) = G'(0) = 0$. The pseudospectral approximation of \eqref{eq:DDE} is of the form 
\begin{subequations}
\begin{equation} \label{eq:ps ode A g}
\frac{d}{dt} \begin{pmatrix}
y_0 \\ y
\end{pmatrix} 
 = \linps \begin{pmatrix}
 y_0 \\ y
 \end{pmatrix}
+ \nonlinps \begin{pmatrix}
y_0 \\ y
\end{pmatrix}
\end{equation}
where the linear map $\linps: \mathbb{R}^{n+1} \to \mathbb{R}^{n+1}$ is given by
\begin{equation} \label{eq:ps A}
\linps \begin{pmatrix}
y_0 \\ y
\end{pmatrix} = \begin{pmatrix}
LP(y_0, y) \\
-y_0 D \textbf{1} + Dy
\end{pmatrix}
\end{equation}
and where the function $\nonlinps$ 
\begin{equation} \label{eq:ps g}
\nonlinps(y_0, y) = \begin{pmatrix}
G(P(y_0, y)) \\ 0
\end{pmatrix}
\end{equation}
satisfies
$\nonlinps (0) = 0$ and $D \nonlinps (0) = 0$.
\end{subequations}
We start by giving a characterization of the eigenvectors of the matrix $\linps$. 

\begin{lemma}[cf. \cite{breda05, dW21}] \label{lem:spectrum ps}
Let $\linps: \mathbb{R} \times \mathbb{R}^n \to \mathbb{R} \times \mathbb{R}^n$ be the linear map defined in \eqref{eq:ps A}. Denote by $D: \mathbb{R}^n \to \mathbb{R}^n$ the matrix defined in \eqref{eq:def D 1} and let $\textbf{1} = (1, \ldots, 1)^T \in \mathbb{R}^n$. Assume that $\lambda \in \mathbb{C}$ is such that 
\[ D - \lambda I \mbox{ is invertible.} \]
Then the following statements hold: 
\begin{enumerate}
\item $\lambda$ is an eigenvalue of $\linps$ if and only if $\Delta_n(\lambda) = 0$, where 
\begin{equation} \label{eq:def Delta n}
\Delta_n(\lambda) := \lambda - LP(1, (D-\lambda I)^{-1}D\textbf{1}));
\end{equation}
\item If $(\zeta, \eta)^T \in \mathbb{R} \times \mathbb{R}^n$ is an eigenvector of $\linps$ with eigenvalue $\lambda$, then $(\zeta, \eta)^T$ is of the form 
\[ \begin{pmatrix}
\zeta \\ \eta
\end{pmatrix} = \zeta \begin{pmatrix}
1 \\ (D - \lambda I)^{-1} D \textbf{1}
\end{pmatrix}. \]
\end{enumerate}
\end{lemma}
\begin{proof}
The matrix $\linps$ has an eigenvalue $\lambda \in \mathbb{C}$ with eigenvector $(\zeta, \eta)^T \in \mathbb{R} \times \mathbb{R}^n$ if and only if
\begin{subequations}
\begin{align}
LP(\zeta, \eta) = \lambda \zeta \label{eq:eigenvector 1} \\
- \zeta D \textbf{1} + D \eta = \lambda \eta \label{eq:eigenvector 2}
\end{align}
\end{subequations}
By assumption, $D-\lambda I$ is invertible; hence \eqref{eq:eigenvector 2} holds if and only if
\begin{equation} \label{eq:proof eigenvalues}
\eta = \zeta(D-\lambda I)^{-1} D \textbf{1}, 
\end{equation}
which proves the second statement of the lemma.  

To prove the first statement of the lemma, we substitute \eqref{eq:proof eigenvalues} into \eqref{eq:eigenvector 1}, which yields
\begin{equation} \label{eq:proof eigenvalues 2}
LP(1, (D-\lambda I)^{-1} D \textbf{1}) ) \zeta = \lambda \zeta. 
\end{equation}
But \eqref{eq:proof eigenvalues} also implies that $(\zeta, \eta) \neq 0$ if and only if $\zeta \neq 0$; so if $(\zeta, \eta)$ is an eigenvector we can cancel $\zeta$ on both sides in \eqref{eq:proof eigenvalues 2} to obtain 
\[ LP(1, (D-\lambda I)^{-1} D \textbf{1}) ) = \lambda. \]
We conclude that $\lambda$ is an eigenvalue of $\linps$ if and only if 
\[ \lambda - LP(1, (D-\lambda I)^{-1} D \textbf{1}) ) = 0,\]
as claimed. 
\end{proof}

We can think of the function $P(1, (D-\lambda)^{-1} D \textbf{1})$ as an approximation of the exponential function $\varepsilon_\lambda(\theta): = e^{\lambda \theta}$; in fact, for fixed $\lambda$ it holds that
\[ \norm{P(1, (D-\lambda)^{-1} D \textbf{1})-\varepsilon_\lambda}_\infty \to 0 \qquad \mbox{as } n \to \infty; \]
where the convergence is uniform for $\lambda$ in compact subsets \cite{breda05} (but we will not use this fact in the rest of this article). 

\medskip

In the next lemma, we use the parametrization method in the context of the ODE \eqref{eq:ps ode A g}. Since the state space of the ODE \eqref{eq:ps ode A g} is the space $\mathbb{R}^{n+1}$, we a-priori expect that we can reformulate the problem of finding a $d$-dimensional unstable manifold as a zero-finding problem on a space of \emph{vector}-valued sequences $S_d^{n+1}$. But we show that, due to the special form of the ODE \eqref{eq:ps ode A g}--\eqref{eq:ps g}, we can reformulute it as a zero finding problem on the space of \emph{scalar}-valued sequences $S_d$, which gives a dimension reduction of the problem. 

\begin{theorem} \label{prop:parametrization_ps}
Consider the DDE \eqref{eq:DDE} satisfying Hypothesis \ref{hyp:DDE} and let the ODE \eqref{eq:ps ode A g}--\eqref{eq:ps g} be its pseudospectral approximation. Suppose that the matrix $\linps: \mathbb{R} \times \mathbb{R}^n \to \mathbb{R} \times \mathbb{R}^n$ has $d$ eigenvalues in the right half of the complex plane, which we gather in a vector as 
\[
\lambda_n = (\lambda_{1, n}, \ldots, \lambda_{d, n}),
\]
and assume that the eigenvalues satisfy the non-resonance condition 
\begin{equation} \label{eq:nonresoncance_ps}
\dotp{\lambda_n, \beta} \not \in \sigma(\linps) 
\end{equation}
for all $|\beta| \geq 2$. Moreover, assume that 
\[ \dotp{\lambda_n, \beta} \not \in \sigma(D) \]
for all $| \beta | \geq 1$. 

Denote by $\hops: \mathcal{S}_d \to \mathcal{S}_d$
the map with the property that 
\begin{equation} \label{eq:H ps ode}
G\left(\sum_{\substack{| \beta | \leq k }} x_\beta P(1, (D-\dotp{\lambda, \beta}I)^{-1}D\textbf{1})) \sigma^\beta \right) = \sum_{\substack{ | \beta | \leq k }} \hops (x)_{\beta} \sigma^\beta + \mathcal{O}(\sigma^{k+1})
\end{equation}
for every sequence $x = (x_\beta)_{\beta \in \mathbb{N}^d}$; 
define the map $\zerops: \mathcal{S}_d \to \mathcal{S}_d $
as 
\begin{equation} \label{eq:F_n_general}
\begin{aligned}
\zerops (x)_\beta = \begin{cases}
x_\beta \qquad &\mbox{if } \beta = 0, \\
x_\beta - \xi_\beta \qquad &\mbox{if } |\beta| = 1, \\
\Delta_n(\dotp{\lambda_n, \beta}) x_\beta - \hops(x)_\beta, \qquad &\mbox{if } | \beta | \geq 2.
\end{cases}
\end{aligned}
\end{equation}
Then, if the sequence $p_n = ( p_{n, \beta} )_{\beta \in \mathbb{N}^m}$ satisfies
\[ 
\zerops(p_n) = 0 \qquad \mbox{ and } \qquad
\sum_{\beta \in \mathbb{N}^2} \left| p_{n, \beta} \right| < \infty,\]
then the graph of the function 
\begin{align*}
\pps&: \mathbb{C}^d \supseteq B(0, 1) \to \mathbb{C}^{n+1} \\
\pps(\sigma) &= \sum_{\beta \in \mathbb{N}^d} p_{n, \beta} \begin{pmatrix}
1 \\
(D-\dotp{\lambda_n, \beta})^{-1} D \textbf{1}
\end{pmatrix}  \sigma^\beta
\end{align*}
is a local unstable manifold for the ODE \eqref{eq:ps ode A g}--\eqref{eq:ps g}. 
\end{theorem}
\begin{proof}
We divide the proof into three steps: 

\medskip
\noindent
\textsc{Step 1:} We first apply Proposition \ref{prop:pm ode} to the ODE \eqref{eq:ps ode A g}--\eqref{eq:ps g}.   
To that end, we denote by $\tilde{\hops}: \mathcal{S}_d \times \mathcal{S}_d^n \to \mathcal{S}_d$
the map such that 
\begin{equation} \label{eq:script H ps ode}
G\left(\sum_{\substack{| \beta | \leq k }} P(a_\beta, b_\beta) \sigma^\beta \right) = \sum_{|\beta| \leq k} \left(\tilde{\hops}(a, b) \right)_\beta \sigma^\beta + \mathcal{O}(|\sigma|^{k+1}).
\end{equation}
Then the function $\nonlinps$ defined in \eqref{eq:ps g} satisfies
\begin{equation*} 
\nonlinps \left( \sum_{|\beta| \leq k} (a_\beta, b_\beta) \sigma^\beta \right) = \begin{pmatrix}
\sum_{|\beta| \leq k} \left(\tilde{\hops}(a, b)\right)_\beta \sigma^\beta \\ 0 
\end{pmatrix}
+ \mathcal{O}(|\sigma|^{k+1}). 
\end{equation*}
For $|\beta| = 1$, we denote by $(\zeta_\beta, \eta_\beta)$ the eigenvector of $\linps$ with eigenvalue $\dotp{\lambda_n, \beta}$; and define the map $\tilde{\zerops}: \mathcal{S}_d \times \mathcal{S}_d^n \to \mathcal{S}_d \times \mathcal{S}_d^n$ as 
\[ \tilde{\zerops}(a, b)_\beta 
=
\begin{cases}
0 \qquad &\mbox{if } \beta = 0 \\
(a_\beta, b_\beta)^T - (\zeta_\beta, \eta_\beta) \qquad &\mbox{if } | \beta | = 1\\
(\dotp{\lambda_n, \beta}I - \linps) \begin{pmatrix}
a_\beta \\ b_\beta
\end{pmatrix}
 = \begin{pmatrix}
 \tilde{\hops}(a, b)_\beta \\ 0
 \end{pmatrix}
 \qquad &\mbox{if } | \beta | \geq 2
\end{cases}
\]
Then Proposition \ref{prop:pm ode} implies that if $(a, b) \in \mathcal{S}_d \times \mathcal{S}_d^n$ is such that 
\[ \tilde{\zerops}(a, b) = 0 \qquad \mbox{and} \qquad \sum_{\beta \in \mathbb{N}^d} \left| (a_\beta, b_\beta) \right| < \infty,  \]
then the graph of the function 
\begin{align*}
\pps: \mathbb{C}^{d} \supseteq B(0, 1) \to \mathbb{C}^{n+1} \\
\pps (\sigma) = \sum_{\beta \in \mathbb{N}^d} \begin{pmatrix}
a_\beta \\ b_\beta
\end{pmatrix}  \sigma^\beta
\end{align*}
is a local unstable manifold for the ODE \eqref{eq:ps ode A g}--\eqref{eq:ps g}. 

\medskip
\noindent
\textsc{Step 2:} We next show that for any sequence $(a, b) \in \mathcal{S}_d \times \mathcal{S}_d^n$ with $\tilde{\zerops}(a, b) = 0$, it holds that 
\begin{equation} \label{eq:blowup proof}
b_\beta = a_\beta (D-\dotp{\lambda, \beta}I)^{-1} D \textbf{1}
\end{equation}
for all $\beta \in \mathbb{N}^d$. 

For $\beta = 0$, the statement \eqref{eq:blowup proof} holds trivially and for $|\beta| = 1$, statement \eqref{eq:blowup proof} follows from Lemma \ref{lem:spectrum ps}.  
For $| \beta| \geq 2$, $\tilde{\zerops}(a, b)_\beta = 0$ is satisfied if and only if 
\[ (\dotp{\lambda_n, \beta}I - \linps) \begin{pmatrix}
a_\beta \\ b_\beta
\end{pmatrix}
 - \begin{pmatrix}
 \tilde{\hops}(a, b)_\beta \\ 0
 \end{pmatrix} = 0,  \]
i.e. if and only if 
\begin{subequations}
\begin{align}
\dotp{\lambda_n, \beta} a_\beta - LP(a_\beta, b_\beta) &= \tilde{\hops}(a, b)_\beta \label{eq:hom eq ps ode 1} \\
a_\beta D \textbf{1} - D b_\beta + \dotp{\lambda_n, \beta} b_\beta &= 0.  \label{eq:hom eq ps ode 2}
\end{align}
\end{subequations}
But since we have assumed that $D-\dotp{\lambda_n, \beta}I$ is invertible for all $\abs{\beta} \geq 1$, \eqref{eq:hom eq ps ode 2} holds 
if and only if \eqref{eq:blowup proof} holds. 

\medskip
\noindent
\textsc{Step 3:} Next, we show that sequence $(a, b) \in \mathcal{S}_d \times \mathcal{S}_d^n$ satisfies $\tilde{\zerops}(a, b) = 0$ if and only if the sequence $a$ satisfies $\zerops (a) = 0$ and the sequence $b$ is given by \eqref{eq:blowup proof}. 

We define the operator $\blowupvectorps: \mathcal{S}_d \to \mathcal{S}_d^n$ via the relation 
\[ (\blowupvectorps (a))_\beta = a_\beta (D-\dotp{\lambda_n, \beta}I)^{-1} D \textbf{1} \qquad \mbox{for} \quad \beta \in \mathbb{N}^d \]
so that the relation \eqref{eq:blowup proof} can be written as $b = \blowupvectorps(a)$. Then substituting \eqref{eq:blowup proof} into \eqref{eq:hom eq ps ode 1} and using the notation \eqref{eq:def Delta n} yields that 
\begin{equation} \label{eq:substitute pm ps ode}
\Delta_n(\dotp{\lambda_n, \beta}) a_\beta - \tilde{\hops}(a, \blowupvectorps (a))_\beta = 0.
\end{equation}
Substituting $b = \blowupvectorps (a)$ in \eqref{eq:script H ps ode} implies that
\[ G\left(\sum_{|\beta| \leq k} a_\beta P(1, (D-\dotp{\lambda_n, \beta}I)^{-1} D \textbf{1}) \sigma^\beta \right) = \sum_{| \beta | \leq k} \tilde{\hops} (a, \blowupvectorps(a))_\beta \sigma^\beta + \mathcal{O}(| \sigma |^{k+1});\]
and comparing this with \eqref{eq:H ps ode} we see that 
\[ \tilde{\hops} (a, \blowupvectorps (a)) = \hops(a). \]
Thus \eqref{eq:substitute pm ps ode} holds if and only if 
\[ \Delta_n(\dotp{\lambda_n, \beta} a_\beta - \hops(a)_\beta = 0, \]
which proves the claim. 
\end{proof}

For Wright's equation
\[ \dot{x}(t) = - \alpha x(t-1)(1+ x(t)) \]
the pseudospectral approximation is given by 
\[ 
\begin{pmatrix}
\dot{y}_0 \\ \dot{y}
\end{pmatrix}
= \begin{pmatrix}
-\alpha y^{(n)} (1+y_0) \\
D y - y_0 D \textbf{1}
\end{pmatrix}
\]
(cf. Example \ref{ex:wright}). 
In the following, we look at the case where the linearisation of the equation around zero has exactly two eigenvalues in the right half of the complex plane, and formulate a zero-finding problem to find the associated unstable manifold.

\begin{corollary} \label{cor:pm_wright_ps}
Consider the ODE 
\begin{equation} \label{eq:ps ode wright}
\begin{pmatrix}
\dot{y}_0 \\ \dot{y}
\end{pmatrix}
= \begin{pmatrix}
-\alpha y^{(n)} (1+y_0) \\
D y - y_0 D \textbf{1}
\end{pmatrix}
\end{equation}
with $y_0 \in \mathbb{R}, \ y \in \mathbb{R}^n$ and parameter $\alpha > 0$. Assume that the linear map 
\[
A_n \begin{pmatrix}
y_0 \\ y
\end{pmatrix}
=
\begin{pmatrix}
- \alpha y^{(n)} \\
Dy - y_0 D \textbf{1}
\end{pmatrix}
\]
has two exactly two eigenvalues $\lambda_{n, +}, \ \lambda_{n, -}$ in the right half of the complex plane, which we collect in a vector $\lambda_n = (\lambda_{n, +}, \lambda_{n, -})$. Moreover, assume that
\[ \dotp{\lambda_n, \beta} \not \in \sigma(D) \]
for all $\beta$ with $|\beta| \geq 1$. 
Define the function 
$\zerops: \mathcal{S}_2 \to \mathcal{S}_2$ as 
\begin{equation} \label{eq:zero_map_wright_ps}
\begin{aligned}
\zerops(x)_\beta = \begin{cases}
x_\beta \qquad &\mbox{if } \beta = 0 \\
x_\beta - \xi_\beta \qquad &\mbox{if } |\beta| = 1 \\
\Delta_n(\dotp{\lambda_n, \beta}) x_\beta + \alpha (x*\blowupscalarps(x))_\beta, \qquad &\mbox{if } | \beta | \geq 2,
\end{cases}
\end{aligned}
\end{equation}
where for $|\beta| = 1$ the constants $\xi_\beta \in \mathbb{C}$ are arbritrary but fixed, 
the operator $\blowupscalarps: \mathcal{S}_2 \to \mathcal{S}_2$ is defined via the relation 
\begin{equation} \label{eq:def_blowup_scalar_ps}
\blowupscalarps (x)_\beta = x_\beta ((D-\dotp{\lambda_n, \beta}I)^{-1}D \textbf{1})_n, \qquad \beta \in \mathbb{N}^2
\end{equation} 
and the function $\Delta_n$ is defined as 
\[
\Delta_n(\lambda) = \lambda + \alpha ((D-\lambda I)^{-1} D \textbf{1})_n.
\]
If $p_n = (p_{n, \beta})_{\beta \in \mathbb{N}^d}$ is a sequence such that $\zerops (p_n) = 0$ and $\sum_{\beta} | p_{n, \beta}| < \infty$, 
then the graph of the function
\[ \pps (\sigma) = \sum_{\beta} p_{n, \beta} \begin{pmatrix}
1 \\ 
(D-\lambda_n I)^{-1} D \textbf{1}
\end{pmatrix} \sigma^\beta\]
is a local unstable manifold around the origin of \eqref{eq:ps ode wright}. 
\end{corollary}
\begin{proof}
Since Wright's equation \eqref{eq:wright} is of the form \eqref{eq:DDE} with functionals $L, \, G$ given by
\[
L \varphi = - \alpha \varphi(-1), \qquad G(\varphi) = -\alpha \varphi(-1) \varphi(0),
\]
the pseudospectral ODE \eqref{eq:ps ode wright} is of the form \eqref{eq:ps ode A g} with the same $L$ and $G$. Hence in this case, the function $\Delta_n$ is given by 
\[
\Delta_n(\lambda) = \lambda + \alpha ((D-\lambda)^{-1} D \textbf{1})_n. 
\]
Moreover, we (formally) compute that 
\begin{align*}
G\left(\sum_{\substack{ \beta \in \mathbb{N}^2 }} x_\beta P(1, (D-\dotp{\lambda_n, \beta}I)^{-1}D\textbf{1})) \sigma^\beta \right) &= - \alpha \left( \sum_{\beta \in \mathbb{N}^2} x_\beta((D-\dotp{\lambda_n, \beta}I)^{-1}D\textbf{1})_n \sigma^\beta \right) \left( \sum_{\beta \in \mathbb{N}^2} x_\beta \sigma^\beta \right) \\
& = - \alpha \left( \sum_{\beta \in \mathbb{N}^2} r_n(x)_\beta \sigma^\beta \right) \left( \sum_{\beta \in \mathbb{N}^2} x_\beta \sigma^\beta \right) \\
& = - \alpha \sum_{\beta \in \mathbb{N}^2} \left( \sum_{\beta_1 + \beta_2 = \beta} r_n(x)_{\beta_1} x_{\beta_2} \right) \sigma^\beta \\
& = - \alpha (x \ast r_n(x))_\beta. 
\end{align*}
Hence $\hops: \mathcal{S}_2 \to \mathcal{S}_2$ in \eqref{eq:H ps ode} is in this case given by 
\[
\hops (x)_\beta = - \alpha (x \ast r_n(x))_\beta.
\]
and from here the claim follows from Proposition \ref{prop:parametrization_ps}. 
\end{proof}

The expression for $\zerops$ in \eqref{eq:F_n_general} involves the expressions $(D - \dotp{\lambda_n, \beta}I)^{-1}$ (namely in the expression for $\Delta_n(\dotp{\lambda_n, \beta})$ and the expression for $H_n$); so in order for the map $\zerops$ to be well-defined, we should first verify that the matrices $D - \dotp{\lambda_n, \beta}I$ are in fact invertible for all $\beta \in \mathbb{N}^d$. 
Lemma \ref{lem:resolvent_estimate_ode} and Corollary \ref{cor:resolvent_D} describe how, for fixed discretization index $n \in \mathbb{N}$, this can be checked by a combination of analytical and numerical methods. 

In the following, we denote by $\lambda_n = (\lambda_{n, 1}, \ldots, \lambda_{n, d})$ a vector of consisting of the eigenvalues of $\linps$ that are in the right half of the complex plane, and with slight abuse of notation we write 
\begin{equation} \label{eq:def_min_re_lambda}
\min(\Re(\lambda_n)) := \min \{ \Re(\lambda_{n, 1}), \ldots, \Re(\lambda_{n, d}) \}. 
\end{equation}

\begin{lemma} \label{lem:resolvent_estimate_ode}
Consider the DDE \eqref{eq:DDE} satisfying Hypothesis \ref{hyp:DDE} and, for fixed $n \in \mathbb{N}$, let the matrix $\linps$ be as defined in \eqref{eq:ps A}. 
Denote by $\lambda_n = (\lambda_{1, n}, \ldots, \lambda_{d, n})$ a vector of consisting of the eigenvalues of $\linps$ that are in the right half of the complex plane. 
Additionally, fix $\epsilon > 0$ and let $M \in \mathbb{N}$ be such that 
\begin{equation} \label{eq:M_treshhold_general}
M >  \frac{\norm{D}+\epsilon}{\min(\Re(\lambda_n))}. 
\end{equation}
Then 
\[
D - \dotp{\lambda_n, \beta}I \quad \mbox{ is invertible for all } \beta \in \mathbb{N}^d \mbox{ with } \abs{\beta} \geq M
\]
and additionally
\begin{align} \label{eq:resolvent_estimate_ode}
\abs*{((D- \dotp{\lambda_n, \beta} I)^{-1} D \textbf{1})_n} \leq \frac{\norm{D} \norm{\textbf{1}}}{\epsilon} \quad \quad \mbox{ for all } \beta \in \mathbb{N}^d \mbox{ with } \abs{\beta} \geq M. 
\end{align}
\end{lemma}
\begin{proof}
Let $\beta \in \mathbb{N}^d$ be such that $\abs{\beta} \geq M$; we first derive a lower bound for the quantity $\abs{\dotp{\lambda_n, \beta}}$. To do so, we first estimate that
\begin{align*}
\Re(\dotp{\lambda_n, \beta}) &= \Re(\lambda_{n, 1}) \beta_1 + \ldots + \Re(\lambda_{n, d}) \beta_d \\
& \geq \min(Re(\lambda_n)) (\beta_1 + \ldots + \beta_d) \\
& = \min(Re(\lambda_n)) \abs{\beta}  \\
& \geq \min(Re(\lambda_n)) M.
\end{align*}
Together with the estimate $\abs{\mu} \geq \Re(\mu)$ for $\mu \in \mathbb{C}$, this implies that 
\begin{equation} \label{eq:estimate_lambeta}
\abs{\dotp{\lambda_n, \beta}} \geq \min(Re(\lambda_n)) M \qquad \mbox{if } | \beta | \geq M. 
\end{equation}
Thus, if $M$ satisfies the inequality \eqref{eq:M_treshhold_general} and $\abs{\beta} \geq M$, then \eqref{eq:estimate_lambeta} implies that
\begin{align*}
\norm*{\frac{D}{\dotp{\lambda_n, \beta}}} &\leq \frac{\norm{D}}{\min(Re(\lambda_n)) M}  \\
& \leq \frac{\norm{D}}{\norm{D} + \epsilon} \\ 
& < 1.
\end{align*} 
From here the Neumann series imply that the matrix
\[ 
\dotp{\lambda_n, \beta}I - D = \dotp{\lambda_n, \beta} \left(I - \frac{D}{\dotp{\lambda_n, \beta}} \right)
\]
is invertible with inverse
\begin{align*}
\left(\dotp{\lambda_n, \beta}I - D \right)^{-1} &= \frac{1}{\dotp{\lambda_n, \beta}} \left(I - \frac{D}{\dotp{\lambda_n, \beta}} \right)^{-1} \\
& = \sum_{k = 0}^\infty \frac{D^k}{\dotp{\lambda_n, \beta}^{k+1}}. 
\end{align*}
This in particular implies that 
\begin{align*}
\abs*{ \left(\left(\dotp{\lambda_n, \beta}I - D \right)^{-1} D \textbf{1} \right)_n} 
& \leq \norm{ \left(\dotp{\lambda_n, \beta}I - D \right)^{-1} D \textbf{1}} \\
& \leq \norm{\sum_{k = 0}^\infty \frac{D^{k+1} \textbf{1}}{\dotp{\lambda_n, \beta}^{k+1}}} \\
& \leq \sum_{k = 0}^\infty \frac{\norm{D}^{k+1} \norm{\textbf{1}}}{\abs{\dotp{\lambda_n, \beta}}^{k+1}}  \\
& \leq \sum_{k = 0}^\infty \frac{\norm{D}^{k+1} \norm{\textbf{1}}}{(\min(\Re(\lambda_n)) M)^{k+1}} 
\end{align*}
where in the last step we used \eqref{eq:estimate_lambeta}. We can rewrite the last quantity using the 
geometric series 
\[
\sum_{k = 0}^\infty r^{k+1} = \frac{r}{1-r}, \qquad \abs{r} < 1
\]
and obtain that
\begin{align*}
\abs*{ \left(\left(\dotp{\lambda_n, \beta}I - D \right)^{-1} D \textbf{1} \right)_n} 
& \leq \frac{\frac{\norm{D}}{\min(\Re(\lambda_n))M}}{1-\frac{\norm{D}}{\min(\Re(\lambda_n))M}} \norm{\textbf{1}} \\
& = \frac{\norm{D}}{\min(\Re(\lambda_n))M-\norm{D}} \norm{\textbf{1}} \\
& \leq \frac{\norm{D} \norm{\textbf{1}}}{\varepsilon}
\end{align*}
where in the last step we used \eqref{eq:M_treshhold_general}. This proves \eqref{eq:resolvent_estimate_ode}.
\end{proof}

\begin{corollary} \label{cor:resolvent_D}
Consider the DDE \eqref{eq:DDE} satisfying Hypothesis \ref{hyp:DDE} and, for fixed $n \in \mathbb{N}$, let the matrix $\linps$ be as defined in \eqref{eq:ps A}. 
Denote by $\lambda_n = (\lambda_{1, n}, \ldots, \lambda_{d, n})$ a vector of consisting of the eigenvalues of $\linps$ that are in the right half of the complex plane. 
Fix $\epsilon > 0$ and let $M \in \mathbb{N}$ be such that \eqref{eq:M_treshhold_general} holds.  

If for each $\beta \in \mathbb{N}^d$ with $\abs{\beta} < M$, the equation
\begin{equation}
\label{eq:invertible_condition}
(D - \dotp{\lambda_n, \beta}I) v = D \textbf{1}
\end{equation}
has a unique solution $v$. Then the matrices $D - \dotp{\lambda_n, \beta}I$ are invertible for all $\beta \in \mathbb{N}^d$ and hence the map $\zerops$ in \eqref{eq:F_n_general} is well-defined. 
\end{corollary}
\begin{proof}
If \eqref{eq:invertible_condition} has a unique solution for each $\beta \in \mathbb{N}^d$ with $\abs{\beta} < M$, then $D-\dotp{\lambda_n, \beta}I$ is invertible for all $\beta \in \mathbb{N}^d$ with $\abs{\beta} < M$. By Lemma \ref{lem:resolvent_estimate_ode}, the map $D-\dotp{\lambda_n, \beta}I$ is invertible for all $\beta \in \mathbb{N}^d$ with $\abs{\beta} \geq M$. Hence $D - \dotp{\lambda_n, \beta}I$ is invertible for all $\beta \in \mathbb{N}^d$ and the claim follows. 
\end{proof}

For a fixed index $n \in \mathbb{N}$ and a fixed vector $\lambda_n \in \mathbb{N}^d$, checking whether the equation \eqref{eq:invertible_condition} is invertible for all $\abs{\beta} < M$ is a finite computatation that canbe rigorously done in, for example, IntLab. For the pseudospectral approximation to Wright's equation with parameter values $\alpha = \placealpha$ and $n = \placen$,
we will perform this computation in Lemma \ref{lem:well-defined_cap} at the end of Subsection \ref{sec:validation_wright_ps}.

\medskip

Lemma \ref{lem:resolvent_estimate_ode} also directly provides a bound on the operator norm of the operator
\begin{equation*} 
\begin{aligned}
\blowupscalarps&: \mathcal{S}_d \to \mathcal{S}_d \\
\blowupscalarps (x)_\beta &= x_\beta ((D-\dotp{\lambda_n, \beta}I)^{-1}D \textbf{1})_n,
\end{aligned}
\end{equation*}
which we will use in Section \ref{sec:validation_wright_ps}:

\begin{lemma} \label{lem:R_n_norm}
Consider the DDE \eqref{eq:DDE} satisfying Hypothesis \ref{hyp:DDE} and, for fixed $n \in \mathbb{N}$, let the matrix $\linps$ be as defined in \eqref{eq:ps A}. 
Denote by $\lambda_n = (\lambda_{1, n}, \ldots, \lambda_{d, n})$ a vector of consisting of the eigenvalues of $\linps$ that are in the right half of the complex plane. 

Suppose that for fixed $\epsilon >0$ and $M \in \mathbb{N}$ satisfying \eqref{eq:M_treshhold_general}, the equation \eqref{eq:invertible_condition} has a unique solution for all $\abs{\beta} \leq M$. 
Then, the map $\blowupscalarps: \mathcal{S}_d \to \mathcal{S}_d$ defined as 
\[
\blowupscalarps (x)_\beta = x_\beta ((D-\dotp{\lambda_n, \beta}I)^{-1}D \textbf{1})_n
\]
is well-defined and it holds that
\begin{equation} \label{eq:R_n_norm}
\norm{\blowupscalarps} \leq \max_{\abs{\beta} < M} \left \{ \abs{((D-\dotp{\lambda_n, \beta}I)^{-1}D\textbf{1})_n}, \frac{\norm{D} \norm{\textbf{1}}}{\epsilon} \right \}.
\end{equation}  
\end{lemma}
\begin{proof}
Let $v \in \mathcal{S}_d$, then 
\begin{align*}
\norm{\blowupscalarps (v)} 
&= \sum_{\beta \in \mathbb{N}^d} \abs{((D-\dotp{\lambda_n, \beta}I)^{-1}D\textbf{1})_n v_\beta } \\
& \leq \max_{\beta \in \mathbb{N}^d} \left \{ \abs{((D-\dotp{\lambda_n, \beta}I)^{-1}D\textbf{1})_n} \right \} \sum_{\beta \in \mathbb{N}^2} \abs{v_\beta} \\
& = \max_{\beta \in \mathbb{N}^d} \left \{ \abs{((D-\dotp{\lambda_n, \beta}I)^{-1}D\textbf{1})_n} \right \}   \norm{v}. 
\end{align*}
Lemma \ref{lem:resolvent_estimate_ode} implies that
\[ \abs{((D-\dotp{\lambda_n, \beta}I)^{-1}D\textbf{1})_n} \leq \frac{\norm{D} \norm{\textbf{1}}}{\epsilon}, \qquad \mbox{if } \abs{\beta} \geq M \]
and hence 
\[ \max_{\beta \in \mathbb{N}^2} \left \{ \abs{((D-\dotp{\lambda_n, \beta}I)^{-1}D\textbf{1})_n} \right \}  = \max_{\abs{\beta} < M} \left \{ \abs{((D-\dotp{\lambda_n, \beta}I)^{-1}D\textbf{1})_n}, \frac{\norm{D} \norm{\textbf{1}}}{\epsilon} \right \}, \]
which proves the claim. 
\end{proof}

\section{Computer Assisted Proof for Wright's equation}
\label{sec:CAP}

In this section, we provide a computer assisted algorithm to compute the distance between the unstable manifold of Wright's equation \eqref{eq:wright} and the unstable manifold of its pseudosepctral approximation \eqref{eq:ps ode wright}. In the accompanying code on \github, we execute this scheme for $\alpha = \placealpha$ and discretization index $n = \placen$, leading to Theorem \ref{thm:distance_intro} (cf. Theorem \ref{thm:distance_conclusion} below). The proof is based upon the Radii Polynomial Theorem; cf. Theorem \ref{thm:radii_polynomial} in Appendix \ref{appendix:contraction_mapping}. For all numerical computations in this paper, we use the MatLab interval arithmetic library IntLab \cite{Ru99a}. 

\subsection{{Computing an approximation of the unstable manifold in the pseudospectral ODE}}
\label{sec:approximatezero_wright_ps}

As a preparatory step, 
we compute interval enclosures of \emph{all} eigenvalues of the linearization of \eqref{eq:ps ode wright} around the origin, i.e. we compute interval enclosures of the eigenvalues of the linear map
\begin{equation} \label{eq:lin_wright_ps}
\begin{pmatrix}
y_0 \\ y
\end{pmatrix}
\mapsto
\begin{pmatrix}
- \alpha y^{(n)} \\
 Dy - y_0 D \textbf{1}
\end{pmatrix} 
\end{equation}
Lemma \ref{lem:spectrum ps} implies that if a complex number $\lambda$ is a zero of the map 
\begin{equation} \label{eq:ce_wright_ps}
\Delta_n(z) = z + \alpha ((D - z I)^{-1} D \textbf{1})_n
\end{equation}
(in which case $D - \lambda I$ is in particular invertible), then $\lambda$ is an eigenvalue of the map \eqref{eq:lin_wright_ps}. In particular, if the function $\Delta_n$ has exactly $n+1$ zeros, then all eigenvalues of \eqref{eq:lin_wright_ps} correspond to zeroes of $\Delta_n$ (i.e. there is no eigenvalue $\lambda$ of \eqref{eq:lin_wright_ps} for which $D - \lambda I$ is not invertible). So if we obtain interval enclosures of the zeroes of $\Delta_n$, and check that there are exactly $n+1$ such zeroes, then we automatically obtain interval enclosures of the eigenvalues of \eqref{eq:lin_wright_ps}. 

We obtain interval enclosures of the zeroes of $\Delta_n$ by first computing its approximate zeroes, i.e. we compute points $\hat{z}$ such that $\Delta_n(\hat{z})$ is small (in practice, Lemma \ref{lem:spectrum ps} implies that numerically computed eigenvalues of \eqref{eq:lin_wright_ps} are good candidates for these approximate zeroes). 
Then for each approximate zero $\hat{z}$ of $\Delta_n$, we use the Radii Polynomial Theorem in Lemma \ref{thm:contraction_mapping_scalar} to obtain an interval enclosure of a `true' zero of $\Delta_n$ close to $\hat{z}$. 
In applying Lemma \ref{thm:contraction_mapping_scalar}, we use that
\begin{align*}
\Delta_n'(z) & = 1 + \alpha\left((D - zI)^{-2}D \mathbf{1}\right)_n \\
\Delta_n''(z) & = 2\alpha\left((D - zI)^{-3}D \mathbf{1}\right)_n. 
\end{align*}
For the case where $\alpha = \placealpha$ and $n = \placen$, computing enclosures of \emph{all} eigenvalues of \eqref{eq:lin_wright_ps} in particular shows that there are exactly two eigenvalues in the right half of the complex plane: 

\begin{lemma} \label{lem:eigenvalues_ps_explicit}
For $\alpha = \placealpha$ and $n = \placen$, the linear map \eqref{eq:lin_wright_ps} has exactly two eigenvalues in the right half of the complex plane, which we denote by $\lambda_{n,+}$ and $\lambda_{n, -}$ and which are enclosed in the intervals
\begin{align*}
\lambda_{n, +} & \in [0.172816002828147 + 1.673686413740504i, 0.172816002828167 + 1.673686413740524i] \\
\lambda_{n, -} & \in [0.172816002828147 - 1.673686413740504i, 0.172816002828167 - 1.673686413740524i]. 
\end{align*}
\end{lemma}

\medskip

Having obtained interval enclosures of the eigenvalues of \eqref{eq:lin_wright_ps}, we next construct an approximate zero of the map $\zerops$ defined in \eqref{eq:zero_map_wright_ps}, i.e. we construct a sequence $\hat{x} \in \mathcal{S}_2$ such that $\norm{\zerops (\hat{x})}$ is small. The identity \eqref{eq:zero_map_wright_ps} implies that a sequence $x \in \mathcal{S}_2$ satisfies $\zerops(x) = 0$ if and only if the equalities
\begin{subequations}
\begin{equation} \label{eq:recursive_xhat_1}
\begin{aligned}
x_\beta = 0 \qquad &\mbox{for } \beta = 0 \\
x_\beta = \xi_\beta \qquad &\mbox{for } |\beta| = 1
\end{aligned}
\end{equation}
and 
\begin{equation} \label{eq:recursive_xhat_2}
x_\beta = \alpha \Delta_n(\dotp{\lambda_n, \beta})^{-1} (x \ast \blowupscalarps (x))_\beta, \qquad \mbox{for } | \beta | \geq 2
\end{equation}
hold. 
\end{subequations}
For fixed $\beta$, the right hand side of the equality \eqref{eq:recursive_xhat_2} only involves $x_{\beta'}$ with $|\beta'| < |\beta|$ and hence we can use it as a recursive relation. 
Hence, we can construct a sequence $\hat{x} \in \mathcal{S}_2$ by: 
\begin{itemize}
\item First fixing scalars $\xi_\beta$ for $|\beta| = 1$ and a natural number $N \in \mathbb{N}$; 
\item For $|\beta| \leq N$, defining the coefficients $\hat{x}_\beta$ via the relations \eqref{eq:recursive_xhat_1}--\eqref{eq:recursive_xhat_2};
\item For $| \beta | > N$, defining $\xi_\beta = 0$. 
\end{itemize}

\begin{remark} \label{rem:xhat}
For $\alpha  = \placealpha$, $n = \placen$ and $N = \placeN$, the above procedure leads to a sequence $\hat{x}$ whose explicit coefficients can be found on the accompanying code on \github. 
\end{remark}

\subsection{Validating the unstable manifold in the pseudospectral ODE}
\label{sec:validation_wright_ps}

We next use the Radii Polynomial Theorem to prove that the map $F_n$ defined in \eqref{eq:zero_map_wright_ps} has a true zero in a neighbourhood of the approximate solution $\hat{x}$.

From \eqref{eq:zero_map_wright_ps}, we compute that
\begin{equation} \label{eq:DF_n}
\begin{aligned}
\left(D\zerops (x)v\right)_\beta = 
\begin{cases}
v_\beta \qquad &\mbox{if } | \beta | \leq 1, \\
\Delta_n (\dotp{\lambda_n, \beta})v_\beta + \alpha( \blowupscalarps (x)*v + \blowupscalarps (v)*x)_\beta \qquad &\mbox{if } \left| \beta \right| \geq 2 ,
\end{cases}
\end{aligned}
\end{equation}
for $x, v \in \mathcal{S}_d$. From here, we compute a numerical approximation of $\pi^N DF_n(\hat{x}) \pi^N$ and denote the result by $\numder_n$. We denote by $\numinv_n$ a numerical a numerical inverse of $\numder_n$ and define the operators $A_n, A_n^\dagger: \mathcal{S}_2 \to \mathcal{S}_2$ as
\begin{subequations}
\begin{align}
\label{eq:A_n}
(A_nv)_\beta & = 
\begin{cases}
(\numinv_n v)_\beta & \abs{\beta} \leq N, \\
\Delta_n(\dotp{\lambda_n, \beta})^{-1}v_\beta & \abs{\beta} > N,
\end{cases} \\
\label{eq:A_n_dagger}
(A_n^\dagger v)_\beta & = 
\begin{cases}
(\numder_n v)_\beta & \abs{\beta} \leq N, \\
\Delta_n(\dotp{\lambda_n, \beta})v_\beta & \abs{\beta} > N. 
\end{cases}
\end{align}
\end{subequations}
For this choice of $A_n$ and $A_n^{\dagger}$, we will provide explicit expressions for the quantities in \eqref{eq:constants_banach_radii} that are used in the Radii Polynomial Theorem. The derived expressions are such that, once we have fixed parameters $\alpha$ and $n$, the quantities can explicitly computed by a computer. 

\begin{lemma}[$Y_0$ bound] \label{lem:Y0_ps}
Fix $\alpha > 0$, $n \in \mathbb{N}$ and let $\zerops$ be as defined in \eqref{eq:zero_map_wright_ps}. Let $\hat{x}  \in \pi_N \cS_2$ be a pseudo-finite sequence; let the operator $A_n$ be defined via \eqref{eq:A_n} and let $\blowupscalarps$ be as defined in \eqref{eq:def_blowup_scalar_ps}. Let 
\begin{equation} \label{eq:Y0 bound ode}
Y_0 :=  
\sum_{\abs{\beta} = 0}^N \norm{(\numinv_n \zerops (\hat{x}))_\beta} + \sum_{\abs{\beta} = N+1}^{2N} \norm{\Delta_n(\dotp{\lambda_n, \beta})^{-1} (\alpha \blowupscalarps (\hat{x})*\hat{x})_\beta}
\end{equation}
then $\norm{A_n \zerops (\hat{x})} \leq Y_0$. 
\end{lemma}
\begin{proof}
Since $\hat{x} \in \pi_N \cS_2$ (i.e. $\hat{x}_\beta = 0$ for $\left| \beta \right| > N$), it holds that 
\[
\Delta_n(\dotp{\lambda_n, \beta}) \hat{x}_\beta = 0 \qquad \mbox{for } \abs{\beta} \geq N +1 
\]
and 
\[
(\hat{x} \ast \blowupscalarps (\hat{x}))_\beta = 0 \qquad \mbox{for } \abs{\beta} \geq 2N + 1. 
\] 
This implies that 
\begin{align*}
(F_n(\hat{x}))_\beta = 
\begin{cases}
(\hat{x} \ast \blowupscalarps (\hat{x}))_\beta \qquad &\mbox{if } N + 1 \leq \abs{\beta} \leq 2 N \\
0 \qquad &\mbox{if } \abs{\beta} \geq 2N + 1.  
\end{cases}
\end{align*}
From here we compute that 
\[ \norm*{A_n \zerops (\hat{x})} \leq \sum_{\abs{\beta} = 0}^N \norm*{(\numinv_n F_n(\hat{x}))_\beta} + \sum_{\abs{\beta} = N+1}^{2N} \norm{\Delta_n(\dotp{\lambda_n, \beta})^{-1} (\alpha \blowupscalarps (\hat{x})*\hat{x})_\beta} \]
and the claim follows. 
\end{proof}

\begin{lemma}[$Z_0$ bound] \label{lem:Z0_ps}
Let the operators $A_n, A_n^\dagger$ be defined via \eqref{eq:A_n}--\eqref{eq:A_n_dagger} and let 
\[ 
Z_0 := \norm*{I - \numinv_n \numder_n } 
\]
then $\norm{I - A_n A_n^\dagger} \leq Z_0$. 
\end{lemma}
\begin{proof}
Let $X_N = \pi_N \cS_2$ and $X_\infty = \pi_\infty \cS_2$; then with respect to the decomposition $X = X_N \oplus X_\infty$, the operators $A, A^\dagger$ are of the form
\begin{equation*} 
A_n = 
\begin{pmatrix}
\numinv_n & 0 \\
0 & \pi_\infty A_n \pi_\infty
\end{pmatrix}, \qquad 
A_n^\dagger = 
\begin{pmatrix}
\numder_n & 0 \\
0 & \pi_\infty A_n^\dagger \pi_\infty
\end{pmatrix}. 
\end{equation*}
From here, we compute that
\[ A_n A_n^\dagger = \begin{pmatrix}
\numinv_n \numder_n & 0 \\
0 & \pi_\infty A_n \pi_\infty A_n^\dagger \pi_\infty
\end{pmatrix}.
\]
The identities \eqref{eq:A_n}--\eqref{eq:A_n_dagger} imply that $\pi_\infty A_n \pi_\infty A_n^\dagger \pi_\infty$ is equal to the identity operator on $X_\infty$. Hence 
\[
I - A_n A_n^\dagger = 
\begin{pmatrix}
I - \numinv_n \numder_n & 0 \\
0 & 0
\end{pmatrix}
\]
and therefore
\[ 
\norm{I - A_n A_n^\dagger} \leq \norm{I - \numinv_n \numder_n}, 
\]
as claimed. 
\end{proof}

In the following, we denote by 
\[
\lambda_n = (\lambda_{n, +}, \lambda_{n, -})
\]
two unstable eigenvalues of the map \eqref{eq:lin_wright_ps} that are complex conjugates of each other, 
so that $\Re (\lambda_{n, +}) = \Re(\lambda_{n, -})$, and write with slight abuse of notation 
\begin{equation}
\label{eq:def_re_lambdan}
\Re(\lambda_n) := \Re (\lambda_{n, +}) = \Re(\lambda_{n, -}). 
\end{equation}
In particular, since the vector $\lambda_n = (\lambda_{n, +}, \lambda_{n, -})$ is a two-dimensional vector, the quantity $\min(\Re(\lambda_n))$ (cf. \eqref{eq:def_min_re_lambda}) is now equal to $\Re(\lambda_n)$, and the condition \eqref{eq:M_treshhold_general} reduces to 
\begin{equation} \label{eq:N_threshhold}
M >  \frac{\norm{D}+\epsilon}{\Re(\lambda_n)}. 
\end{equation}

\begin{lemma}
\label{lem:Delta_n_inverse}
Fix $n \in \mathbb{N}$, $\alpha > 0$, let $\lambda_n = (\lambda_{n, +}, \lambda_{n, -})$ be a vector of two unstable eigenvalues of \eqref{eq:lin_wright_ps} which are complex conjugates and let $\Re(\lambda_n)$ be as in \eqref{eq:def_re_lambdan}. 

Fix $\epsilon > 0$ and let $M \in \mathbb{N}$ be such that 
\begin{equation} \label{eq:M_treshold_delta}
M > \max \left \lbrace \frac{\norm{D}+\epsilon}{\Re(\lambda_n)}, \, \frac{\alpha \norm{D} \norm{\textbf{1}}}{\Re(\lambda_n) \epsilon} \right \rbrace, 
\end{equation}
then
\begin{align} \label{eq:Delta_n_estimate}
\max_{\abs{\beta} \geq M} \left \{ \frac{1}{\abs{\Delta_n(\dotp{\lambda_n, \beta})}} \right \} \leq \frac{1}{\Re(\lambda_n)M - \frac{\alpha \norm{D} \norm{\textbf{1}}}{\epsilon}}. 
\end{align}
\end{lemma}
\begin{proof}
We first rewrite the expression for $\Delta_n(\dotp{\lambda_n, \beta})$ (cf. \eqref{eq:ce_wright_ps}) as
\begin{align*}
\Delta_n(\dotp{\lambda_n, \beta}) 
&= \dotp{\lambda_n, \beta} - \alpha \left( \left(\dotp{\lambda_n, \beta}I-D\right)^{-1} D \textbf{1} \right)_n. 
\end{align*}
from where the reverse triangle inequality implies that
\begin{align} \label{eq:delta_n_est_step}
\abs{\Delta_n(\dotp{\lambda_n, \beta})} \geq \abs*{ \abs{\dotp{\lambda_n, \beta}} - \abs*{\alpha \left( \left(\dotp{\lambda_n, \beta}I-D\right)^{-1} D \textbf{1} \right)_n}}. 
\end{align}
For all $\beta \in \mathbb{N}^2$ with $\abs{\beta} \geq M$, it holds that \[
\abs{\dotp{\lambda_n, \beta}} \geq \Re(\lambda_n) M
\]
(cf. \eqref{eq:estimate_lambeta}). Moreover, if $M$ satisfies the inequality \eqref{eq:M_treshold_delta}, then it also satisfies the inequality \eqref{eq:N_threshhold} and Lemma \ref{lem:resolvent_estimate_ode} implies that 
\[
\abs*{((\dotp{\lambda_n, \beta} I - D)^{-1} D \textbf{1})_n} \leq \frac{\norm{D} \norm{\textbf{1}}}{\epsilon}. 
\]
Hence it follows that 
\[ \abs{\dotp{\lambda_n, \beta}} - \abs*{\alpha \left( \left(\dotp{\lambda_n, \beta}I-D\right)^{-1} D \textbf{1} \right)_n} \geq \Re(\lambda_n)M - \frac{\alpha \norm{D} \norm{\textbf{1}}}{\epsilon}. \]
Due to the assumption \eqref{eq:M_treshold_delta}, the right hand side of the above inequality is positive; together with \eqref{eq:delta_n_est_step} this implies that 
\begin{align*}
\abs{\Delta_n(\dotp{\lambda_n, \beta})} 
& \geq \Re(\lambda_n)M - \frac{\alpha \norm{D} \norm{\textbf{1}}}{\epsilon} > 0
\end{align*}
for all $\abs{\beta} \geq M$.  
It now follows that  
\begin{align*}
\max_{\abs{\beta} \geq M} \left \{ \frac{1}{\abs{\Delta_n(\dotp{\lambda_n, \beta})}} \right \}  \leq \frac{1}{\Re(\lambda_n)M - \frac{\alpha \norm{D} \norm{\textbf{1}}}{\epsilon}}, 
\end{align*}
as claimed. 
\end{proof}

\begin{lemma}[$Z_1$ bound]
\label{lem:Z1_ps}
Fix $n \in \mathbb{N}$, $\alpha > 0$, let $\hat{x} \in \pi_N \cS_2$ be a pseudo-finite sequence and let $\lambda_n = (\lambda_{n, +}, \lambda_{n, -})$ be a vector of two unstable eigenvalues of \eqref{eq:lin_wright_ps} which are complex conjugates; let $\Re(\lambda_n)$ be as in \eqref{eq:def_re_lambdan}. Fix $\epsilon > 0$ and let $M \in \mathbb{N}$ be such that \eqref{eq:M_treshold_delta} holds. Then the constant
\begin{equation}
\label{eq:wright_Z1_ode}
\begin{aligned}
Z_1 : = &\norm*{\numinv_n \left(\pi_N D\zerops(\hat{x}) \pi_N - \numder_n \right)}  \\
&+ \alpha \max_{N < \abs{\beta} \leq M} \left \{ \frac{1}{\abs{\Delta_n(\dotp{\lambda_n, \beta})}}, \frac{1}{\Re(\lambda_n)M - \frac{\alpha \norm{D} \norm{\textbf{1}}}{\epsilon}} \right \} \\
&\times  \left(\norm{\blowupscalarps (\hat{x})} + \norm{\hat{x}} \max_{0 \leq \abs{\beta} \leq M} \left \{ \abs{((D-\dotp{\lambda_n, \beta}I)^{-1}D\textbf{1})_n}, \frac{\norm{D} \norm{\textbf{1}}}{\epsilon} \right \} \right) 
\end{aligned}
\end{equation}
has the property that
\[ 
\norm{A_n(D \zerops (\hat{x}) - A_n^\dagger)} \leq Z_1. 
\]
\end{lemma}
\begin{proof}
For any sequence $x \in \mathcal{S}_2$, it holds that $\norm{x} \leq \norm{\pi_N x} + \norm{\pi_\infty x}$; 
with $x =  A_n \left(D\zerops (\hat{x}) - A^\dagger_n\right)$, this implies that 
\begin{equation} \label{eq:Z1_ps_step1}
\norm*{\left(A_n \left(D\zerops (\hat{x}) - A_n^\dagger\right)\right)} \leq 
\norm*{\pi_N \left(A_n \left(D \zerops(\hat{x}) - A_n^\dagger\right)\right)} + 
\norm*{\pi_\infty \left(A_n \left(D \zerops (\hat{x}) - A_n^\dagger\right)\right)}
\end{equation}
From the definitions of $A_n$ and $A_n^\dagger$ in \eqref{eq:A_n}--\eqref{eq:A_n_dagger} and from the epxression for $DF_n(\hat{x})$ in \eqref{eq:DF_n} it follows that 
\begin{align*}
\pi_N A_n = \numinv_n \pi_N, \qquad 
\pi_N A_n^\dagger = \numder_n \pi_N, \qquad 
\pi_N DF_n(\hat{x}) = \pi_N DF_n(\hat{x}) \pi_N
\end{align*}
so that we can rewrite the first term on the right hand side of \eqref{eq:Z1_ps_step1} as 
\begin{align} \label{eq:Z1_ps_step2}
\norm*{\pi_N \left(A_n \left(D \zerops(\hat{x}) - A_n^\dagger\right)\right)} = \norm*{\numinv_n \left(\pi_N D \zerops(\hat{x}) \pi_N - \numder_n \right)}. 
\end{align}
To bound the second term on the right hand side of \eqref{eq:Z1_ps_step1}, we let $v \in \mathcal{S}_2$ be a sequence and compute that  
\begin{align*}
\left(\pi_\infty \left(A_n \left(D \zerops (\hat{x}) - A_n^\dagger\right)\right)v\right)_\beta &= 
\left( \pi_\infty A_n D \zerops (\hat{x}) v \right)_\beta - \left( \pi_\infty A_n A_n^\dagger v_\beta \right)_\beta \\
&= \Delta_n^{-1} ( \dotp{\lambda_n, \beta}) \left[ \Delta_n(\dotp{\lambda_n, \beta}) v_\beta + \alpha( \blowupscalarps (\hat{x})*v + \blowupscalarps (v)*\hat{x})_\beta \right]  \\  &\qquad - \Delta_n^{-1}(\dotp{\lambda_n, \beta}) \Delta(\dotp{\lambda_n, \beta}) v_\beta \\
& = \alpha \Delta_n^{-1} ( \dotp{\lambda_n, \beta}) (\blowupscalarps (\hat{x})*v + \blowupscalarps (v)*\hat{x})_\beta. 
\end{align*}
This implies that 
\begin{align*}
\norm{\pi^\infty \left(A_n \left(D\zerops(\hat{x}) - A_n^\dagger\right)\right)v} & = 
\sum_{\abs*{\beta} > N} \abs*{\alpha \Delta_n\left(\langle \lambda_n, \beta \rangle \right)^{-1} \left(\blowupscalarps (\hat{x})*v + \blowupscalarps (v)*\hat{x}\right)_\beta} \\
& \leq \alpha \sup_{\abs*{\beta} > N} \setof*{\abs*{\Delta_n(\langle \lambda_n, \beta \rangle)^{-1}}} \left(\sum_{\abs*{\beta} > N}\abs*{(\blowupscalarps(\hat{x})*v)_\beta} + \abs*{(\blowupscalarps(v)*\hat{x})_\beta}\right) \\
& \leq \alpha \sup_{\abs*{\beta} > N} \setof*{\abs*{\Delta_n(\langle \lambda_n, \beta \rangle)^{-1}}} \left(\norm*{\blowupscalarps(\hat{x})*v} + \norm*{\blowupscalarps(v)*\hat{x}}\right) \\
&  \leq \alpha \sup_{\abs*{\beta} > N} \setof*{\abs*{\Delta_n(\langle \lambda_n, \beta \rangle)^{-1}}} \left(\norm*{\blowupscalarps(\hat{x})}\norm{v} + \norm*{\blowupscalarps(v)}\norm*{\hat{x}}\right)  \\
&  \leq \alpha \sup_{\abs*{\beta} > N} \setof*{\abs*{\Delta_n(\langle \lambda_n, \beta \rangle)^{-1}}} \left(\norm*{\blowupscalarps (\hat{x})} \norm{v} + \norm{\blowupscalarps} \norm{v} \norm{\hat{x}} \right). 
\end{align*}
From Lemma \ref{lem:Delta_n_inverse} it follows that
\[
\sup_{\abs*{\beta} > N} \setof*{\abs*{\Delta_n(\langle \lambda_n, \beta \rangle)^{-1}}} = \max_{N < \abs{\beta} \leq M} \left \{ \frac{1}{\abs{\Delta_n(\dotp{\lambda_n, \beta})}}, \frac{1}{\Re(\lambda_n)M - \frac{\alpha \norm{D} \norm{\textbf{1}}}{\epsilon}} \right \}  
\]
and together with Lemma \ref{lem:R_n_norm} this implies that 
\begin{align*}
\norm{\pi^\infty \left(A_n \left(D\zerops(\hat{x}) - A_n^\dagger\right)\right)} =  &\alpha \max_{N < \abs{\beta} \leq M} \left \{ \frac{1}{\abs{\Delta_n(\dotp{\lambda_n, \beta})}}, \frac{1}{\Re(\lambda_n)M - \frac{\alpha \norm{D} \norm{\textbf{1}}}{\epsilon}} \right \} \\
&\times  \left(\norm{\blowupscalarps (\hat{x})} + \norm{\hat{x}} \max_{0 \leq \abs{\beta} \leq M} \left \{ \abs{((D-\dotp{\lambda_n, \beta}I)^{-1}D\textbf{1})_n}, \frac{\norm{D} \norm{\textbf{1}}}{\epsilon} \right \} \right) .
\end{align*}
Combining this with estimate \eqref{eq:Z1_ps_step1} and equality \eqref{eq:Z1_ps_step2} yields the desired estimate. 
\end{proof}

\begin{lemma}[$Z_2$ bound] \label{lem:Z2_ps}
Fix $n \in \mathbb{N}$, $\alpha > 0$,  let $\lambda_n = (\lambda_{n, +}, \lambda_{n, -})$ be a vector of two unstable eigenvalues of \eqref{eq:lin_wright_ps} which are complex conjugates and let $\hat{x} \in \pi_N \cS_2$ be a pseudo-finite sequence. Additionally fix $\epsilon > 0$ and let $M \in \mathbb{N}$ be such that \eqref{eq:M_treshold_delta} holds. Define the constant
\begin{equation*}
\begin{aligned}
Z_2 = 2 \alpha  &\max_{N< \abs*{\beta} \leq M} \setof*{ \abs*{\Delta_n(\langle \lambda_n, \beta \rangle )^{-1}}, \frac{1}{\Re(\lambda_n) M - \frac{\alpha \norm{D} \norm{\textbf{1}}}{\epsilon}},  \norm{\numinv_n}} \\
\times 
 &\max_{\abs{\beta} \leq M} \left \{ \abs{((D-\dotp{\lambda_n, \beta}I)^{-1}D \textbf{1})_n}, \frac{\norm{D} \norm{\textbf{1}}}{\epsilon} \right \}
\end{aligned}. 
\end{equation*}
Then the inequality 
\[  \norm*{A_n\left(D\zerops (\hat{x}+y) - D \zerops (\hat{x})\right)} \leq Z_2 \norm{y} \]
holds for all $y \in \mathcal{S}_2$. 
\end{lemma}
\begin{proof}
We first estimate the operator norm of $A_n$. We fix a sequence $v \in \mathcal{S}_2$ and use the definition \eqref{eq:A_n} to derive that
\begin{align*}
\norm*{A_nv} & = \sum_{\abs*{\beta} \leq N} \abs*{(\numinv_n v)_\beta} + \sum_{\abs*{\beta} > N} \abs*{\Delta_n (\langle \lambda_n, \beta \rangle )^{-1}v_\beta} \\
& \leq \norm*{\numinv_n} \sum_{\abs*{\beta} \leq N} \abs*{v_\beta} + \sup_{\abs*{\beta} > N} \setof*{\Delta_n(\langle \lambda_n, \beta \rangle )} \sum_{\abs*{\beta} > N} \abs*{v_\beta} \\
&\leq \max_{\abs*{\beta} > N} \setof*{\Delta_n(\langle \lambda_n, \beta \rangle ), \norm{\numinv_n}}  \sum_{\beta \in \mathbb{N}^2} \abs*{v_\beta}  \\
& \leq \max_{N< \abs*{\beta} \leq M} \setof*{ \abs*{\Delta_n(\langle \lambda_n, \beta \rangle )^{-1}}, \frac{1}{\Re(\lambda_n) M - \frac{\alpha \norm{D} \norm{\textbf{1}}}{\epsilon}},  \norm{\numinv_n}} \norm{v}
\end{align*}
where in the last step we used Lemma \ref{lem:Delta_n_inverse}. 
We conclude that we can estimate the operator norm of $A_n$ as 
\begin{equation} \label{eq:norm_An}
\norm*{A_n} \leq \max_{N< \abs*{\beta} \leq M} \setof*{ \abs*{\Delta_n(\langle \lambda_n, \beta \rangle )^{-1}}, \frac{1}{\Re(\lambda_n) M - \frac{\alpha \norm{D} \norm{\textbf{1}}}{\epsilon}},  \norm{\numinv_n}}.
\end{equation}
We next fix a $y\in \mathcal{S}_2$ and a sequence $v \in \mathcal{S}_2$ with $\norm{v} = 1$; 
and compute that
\begin{align*}
(D\zerops (\hat{x}+y)v - D\zerops (\hat{x})v)_\beta = \begin{cases}
0 \qquad &\mbox{if } | \beta | \leq 1 \\
\alpha( \blowupscalarps (y)*v + \blowupscalarps (v)*y)_\beta \qquad &\mbox{if } \left| \beta \right| \geq 2
\end{cases}
\end{align*}
which implies that 
\begin{align*}
\norm{(D\zerops (\hat{x}+y)v - D\zerops (\hat{x})v)} 
& \leq 2 \alpha \norm{\blowupscalarps} \norm{y} \norm{v}  \\
& \leq 2 \alpha \norm{y} \norm{v} \max_{\abs{\beta} \leq M} \left \{ \abs{((D-\dotp{\lambda_n, \beta}I)^{-1}D \textbf{1})_n}, \frac{\norm{D} \norm{\textbf{1}}}{\epsilon} \right \}
\end{align*}
where in the last step we have used Lemma \ref{lem:R_n_norm}. Combining the last inequality with the estimate \eqref{eq:norm_An} yields the statement of the lemma. 
\end{proof}

In the accompanying code, we fix $\alpha = \placealpha$, $n = \placen$ and let $\lambda_{n} = (\lambda_{n, +}, \lambda_{n, -})$ be as in Lemma \ref{lem:eigenvalues_ps_explicit}. Moreover, we fix 
\[
\epsilon = \placeeps \approx 51.2507
\]
and let 
\[
M = \placeM
\]
so that the pair $\epsilon, M$ satisfies the condition \eqref{eq:M_treshold_delta} (and hence in particular also satisfies the condition \eqref{eq:N_threshhold}). Using IntLab, we establish that for each $\beta \in \mathbb{N}^2$ with $\abs{\beta} \leq M$, the equation
\[
(D - \dotp{\lambda_n, \beta}I) v = D \textbf{1}
\]
has a unique solution $v$; from there Corollary \ref{cor:resolvent_D} implies that

\begin{lemma} \label{lem:well-defined_cap}
For $\alpha = 2$, $n = \placen$ and $\lambda_{n} = (\lambda_{n, +}, \lambda_{n, -})$ as in Lemma \ref{lem:eigenvalues_ps_explicit}, the matrices 
\[
D - \dotp{\lambda_n, \beta}I
\]
are invertible for all $\beta \in \mathbb{N}^2$; hence the map $\zerops$ in \eqref{eq:zero_map_wright_ps} is well-defined. 
\end{lemma}

In the accompanying code on \github, we take $\hat{x} \in \pi_N \cS_2$ as in Remark \ref{rem:xhat} and explicitly compute interval enclosures of the expressions for $Y_0, Z_0, Z_1$ and $Z_2$ from Lemma \ref{lem:Y0_ps}--\ref{lem:Z2_ps}. We then verify that the polynomial \[
p(r) = Z_2 r^2 - (1 - Z_0 - Z_1) r + Y_0
\]
has a zero on the positive real line. By means of the Radii Polynomial Theorem (Theorem \ref{thm:radii_polynomial}), this implies that: 

\begin{proposition} \label{prop:r_ps}
Let $\alpha = 2$ and $n = \placen$; let $\lambda_{n} = (\lambda_{n, +}, \lambda_{n, -})$ be as in Lemma \ref{lem:eigenvalues_ps_explicit} and let $\hat{x} \in \pi_n \cS_2$ be the pseudo-finite sequence as constructed in Remark \ref{rem:xhat}. 
Then for 
\[
r_{\text{\tt PSA}} = 1.110565565384011 \times 10^{-13}, 
\]
the map $\zerops$ defined in \eqref{eq:zero_map_wright_dde} has a zero on the ball $B(\hat{x}, r_{\text{\tt PSA}})$. 
\end{proposition}

\subsection{Validating the unstable manifold in the DDE}
\label{sec:validation_wright_dde}

We next use the Radii Polynomial Theorem to prove that the map $F$ defined in \eqref{eq:zero_map_wright_dde} has a true zero in a neighbourhood of the approximate solution $\hat{x}$.

We recall that for $\frac{\pi}{2} < \alpha  < \frac{5 \pi}{2}$, the function \eqref{eq:ce wright} has exactly two zeroes in the right half of the complex plane; as a preparatory step, we obtain interval enclosures of these two zeroes by first computing a numerical approximation $\hat{z}$ for each of them  
and subsequently applying the the Radii Polynomial Theorem (Lemma \ref{thm:contraction_mapping_scalar}). For $\alpha = \placealpha$, we do this in the accompanying \github \, code and  
collect the result in the following lemma: 

\begin{lemma} \label{lem:eigenvalues_dde_explicit}
For $\alpha = \placealpha$, the two zeroes $\lambda_{+}, \lambda_{-}$ of \eqref{eq:ce wright} in the right half of the complex plane are enclosed in the intervals
\begin{align*}
\lambda_{+} & \in [0.172816002840000 + 1.673686413740842i, 0.172816002840000 + 1.673686413740843i] \\
\lambda_{-} & \in [0.172816002840000 - 1.673686413740842i, 0.172816002840000 - 1.673686413740843i]
\end{align*}
\end{lemma}

\medskip

Next we compute that the derivative of the map $\zerodde$ defined in \eqref{eq:zero_map_wright_dde} is given by 
\begin{equation} \label{eq:DF}
\begin{aligned}
\left(D\zerodde (x)v\right)_\beta = 
\begin{cases}
v_\beta \qquad &\mbox{if } | \beta | \leq 1, \\
\Delta (\dotp{\lambda, \beta})v_\beta + \alpha( \blowupscalardde (x)*v + \blowupscalardde (v)*x)_\beta \qquad &\mbox{if } \left| \beta \right| \geq 2,
\end{cases}
\end{aligned}
\end{equation}
for $x, v \in \mathcal{S}_2$. 
Using formula \eqref{eq:DF}, we compute a numerical approximation of $\pi^N DF(\hat{x}) \pi^N$ and we denote this numerical approximation by $\numder$. We also compute a numerical inverse to $\numder$ and denote the result by $\numinv$. 
Next we define the operators $A, A^\dagger: \cS_2 \to \cS_2$ by the formulas
\begin{subequations}
\begin{align}
\label{eq:A}
(Av)_\beta & = 
\begin{cases}
(\numinv v)_\beta & \abs{\beta} \leq N, \\
\Delta(\dotp{\lambda, \beta})^{-1}v_\beta & \abs{\beta} > N,
\end{cases} \\
\label{eq:A_dagger}
(A^\dagger v)_\beta & = 
\begin{cases}
(\numder v)_\beta & \abs{\beta} \leq N, \\
\Delta(\dotp{\lambda, \beta})v_\beta & \abs{\beta} > N. 
\end{cases}
\end{align}
\end{subequations}
For this choice of $A$ and $A^{\dagger}$, we will provide explicit expressions for the quantities in \eqref{eq:constants_banach_radii} that are used in the Radii Polynomial Theorem. We stress that,  
once we have fixed a parameter $\frac{\pi}{2} < \alpha < \frac{5 \pi}{2}$ and a finite sequence $\hat{x} \in \pi_N \cS_2$, 
all the obtained quantities are explicitly computable on a computer. 

\begin{lemma}[$Y_0$ bound] \label{lem:Y0_dde}
Let $\zerodde$ be as defined in \eqref{eq:zero_map_wright_dde}, let $\hat{x}  \in \pi_N \cS_2$ be a pseudo-finite sequence and let the operator $A$ be defined via \eqref{eq:A}. Let 
\begin{equation} \label{eq:Y0 bound}
Y_0 :=  
\sum_{\abs{\beta} = 0}^N \norm{(\numinv \zerodde (\hat{x}))_\beta} + \sum_{\abs{\beta} = N+1}^{2N} \norm{\Delta(\dotp{\lambda, \beta})^{-1} (\alpha \blowupscalardde (\hat{x})*\hat{x})_\beta}
\end{equation}
then $\norm{A\zerodde (\hat{x})} \leq Y_0$. 
\end{lemma}
\begin{proof}
Since $\hat{x} \in \pi_N \cS_2$ (i.e. $\hat{x}_\beta = 0$ for $\left| \beta \right| > N$), it holds that 
\[
\Delta(\dotp{\lambda, \beta}) \hat{x}_\beta = 0 \qquad \mbox{for } | \beta | \geq N +1 
\]
and it additionally holds that  
\[
(\hat{x} \ast \blowupscalardde (\hat{x}))_\beta = 0 \qquad \mbox{for } | \beta | \geq 2N + 1. 
\]
This implies that 
\begin{align*}
F(\hat{x})_\beta = 
\begin{cases}
(\hat{x} \ast \blowupscalardde (\hat{x}))_\beta \qquad &\mbox{if } N+1 \leq | \beta | \leq 2N, \\
0 \qquad &\mbox{if } | \beta | \geq 2N+1.
\end{cases}
\end{align*}
From here, it follows that
\[ \norm*{A\zerodde(\hat{x})} \leq \sum_{\abs{\beta} = 0}^N \norm*{(\numinv \zerodde (\hat{x}))_\beta} + \sum_{\abs{\beta} = N+1}^{2N} \norm{\Delta(\dotp{\lambda, \beta})^{-1} (\alpha \blowupscalardde (\hat{x})*\hat{x})_\beta} \]
and the claim follows. 
\end{proof}

\begin{lemma}[$Z_0$ bound] Let the operators $A, A^\dagger$ be defined via \eqref{eq:A}--\eqref{eq:A_dagger}
and let 
\[ 
Z_0 := \norm*{I - \numinv \numder } . 
\]
Then $\norm{I - A A^\dagger} \leq Z_0$. 
\end{lemma}
\begin{proof}
Let $X_N = \pi_N \cS_2$ and $X_\infty = \pi_\infty \cS_2$; then, with respect to the decomposition $X = X_N \oplus X_\infty$, the operators $A, A^\dagger$ are of the form
\begin{equation} \label{eq:A diagonal}
A = 
\begin{pmatrix}
\numinv & 0 \\
0 & \pi_\infty A \pi_\infty
\end{pmatrix}, \qquad 
A^\dagger = 
\begin{pmatrix}
\numder & 0 \\
0 & \pi_\infty A^\dagger \pi_\infty
\end{pmatrix}. 
\end{equation}
This implies that 
\[ A A^\dagger = \begin{pmatrix}
\numinv \numder & 0 \\
0 & \pi_\infty A \pi_\infty A^\dagger \pi_\infty
\end{pmatrix}.
\]
The identities \eqref{eq:A}--\eqref{eq:A_dagger} imply that $\pi_\infty A \pi_\infty A^\dagger \pi_\infty$ is equal to the identity operator on $X_\infty$. Hence 
\[
I - A A^\dagger = 
\begin{pmatrix}
I - \numinv \numder & 0 \\
0 & 0
\end{pmatrix}
\]
and therefore
\[ 
\norm{I - A A^\dagger} \leq \norm{I - \numinv \numder }, 
\]
as claimed. 
\end{proof}

We next let
\[
\lambda = (\lambda_{+}, \lambda_{-})
\]
be the vector that consists of the two roots of zeroes of the function \eqref{eq:ce wright} in the right half of the complex plane (cf. Lemma \ref{lem:eigenvalues_dde_explicit}). Since the numbers $\lambda_{+}$ and  $\lambda_{-}$ are complex conjugates, it holds that $\Re(\lambda_{+}) = \Re(\lambda_{-})$. In the following, we will slightly abuse notation and write 
\begin{equation} \label{eq:def_re_lambda}
\Re(\lambda) : = \Re(\lambda_{+}) = \Re(\lambda_{-}). 
\end{equation}

\begin{lemma}
\label{lem:delta_bound}
Fix $ \frac{\pi}{2} < \alpha < \frac{5 \pi}{2}$, let $\lambda = (\lambda_{+}, \lambda_{-})$ be the vector of zeroes of \eqref{eq:ce wright} in the right half of the complex plane and let $\Re(\lambda)$ be as in \eqref{eq:def_re_lambda}. If $M \in \mathbb{N}$ is such that 
\begin{equation} \label{eq:N_treshhold_dde}
\Re(\lambda) M - \alpha e^{-\Re(\lambda)M} > 0,
\end{equation}
then it holds that 
\begin{equation} \label{eq:estimate_delta_inverse}
 \sup_{\abs*{\beta} > M} \setof*{\abs{\Delta(\langle \lambda, \beta \rangle )^{-1}}} \leq \frac{1}{\Re(\lambda) M - \alpha e^{-\Re(\lambda)M}}. 
\end{equation}
\end{lemma}
\begin{proof}
The estimate $\abs{\mu} \geq \abs{\Re(\mu)}$ for $\mu \in \mathbb{C}$ in particular implies that 
\begin{equation} \label{eq:estimate_delta_step}
\begin{aligned}
\abs*{\Delta \langle \lambda, \beta \rangle } &\geq \abs{\Re\left(\Delta \langle \lambda, \beta \rangle \right)}
\end{aligned}
\end{equation}
Since the eigenvalues $\lambda_{+}, \lambda_{-}$ are complex conjugates of each other, it holds that 
\[
\Re(\dotp{\lambda, \beta}) = \Re(\lambda) \abs{\beta} \qquad \mbox{and} \qquad \Im(\dotp{\lambda, \beta}) = \Im(\lambda_+) (\beta_1 - \beta_2) 
\]
for every $\beta \in \mathbb{N}^2$. 
For $\beta \in \mathbb{N}^2$ with $\abs{\beta} \geq M$, this implies that 
\begin{align*}
\Re\left(\Delta \langle \lambda, \beta \rangle \right) & = \Re(\dotp{\lambda, \beta}) - \alpha \Re\left(e^{-\dotp{\lambda, \beta}} \right)\\ 
& = \Re(\lambda) \abs{\beta} - \alpha e^{-\Re(\lambda) \abs{\beta}} \cos(\Im(\lambda_+) (\beta_2 - \beta_1)) \\
& \geq \Re(\lambda) \abs{\beta} - \alpha e^{-\Re(\lambda) \abs{\beta}} \\
& \geq \Re(\lambda) M - \alpha e^{-\Re(\lambda) M}. 
\end{align*}
Thus, if \eqref{eq:N_treshhold_dde} holds, $\Re\left(\Delta \langle \lambda, \beta \rangle \right) \geq 0$ and we deduce from \eqref{eq:estimate_delta_step} that 
\begin{align*}
\abs*{\Delta \langle \lambda, \beta \rangle } &\geq \Re(\Delta(\dotp{\lambda, \beta})) \\
& \geq \Re(\lambda) M - \alpha e^{-\Re(\lambda) M}
\end{align*}
which implies the claim.
\end{proof}

In practice, we find that the inequality \eqref{eq:N_treshhold_dde} is satisfied for relatively low values of $M$. In particular, if we let $N \in \mathbb{N}$ be such that the computed approximation $\hat{x}$ satisfies $\hat{x} \in \pi_N S_2$, then we practically find that the inequality is satisfied for $M = N$. Motivated by this, we will in the following apply Proposition \ref{lem:delta_bound} with $M = N$. 

\begin{lemma}[$Z_1$ bound]
\label{lem:wright_Z1}
Let $\zerodde$ be as defined in \eqref{eq:zero_map_wright_dde}, let $\hat{x}  \in \pi_N \cS_2$ be a pseudo-finite sequence and let the operators $A, A^\dagger$ be defined via \eqref{eq:A}--\eqref{eq:A_dagger}. Moreover, let $\Re(\lambda)$ be as in \eqref{eq:def_re_lambda}, let $N \in \mathbb{N}$ be satisfying \eqref{eq:N_treshhold_dde} and define
\begin{equation}
\label{eq:wright_Z1}
Z_1 := \frac{2 \alpha \norm{\hat{x}}}{\Re(\lambda)N - \alpha e^{-\Re(\lambda)N}} + \norm*{\numinv \left(\pi_N D \zerodde (\hat{x}) \pi_N - \numder \right)}.
\end{equation}
Then
\[ 
\norm{A(D \zerodde (\hat{x}) - A^\dagger)} \leq Z_1. 
\]
\end{lemma}
\begin{proof}
For any sequence $x \in \mathcal{S}_2$, it holds that $\norm{x} \leq \norm{\pi_N x} + \norm{\pi_\infty x}$; 
with $x =  A \left(D\zerodde (\hat{x}) - A^\dagger\right)$, this gives that 
\begin{equation} \label{eq:z1_dde_step1}
\norm*{A \left(D\zerodde (\hat{x}) - A^\dagger\right)} \leq 
\norm*{\pi_N \left(A \left(D\zerodde (\hat{x}) - A^\dagger\right)\right)} + 
\norm*{\pi_\infty \left(A \left(D\zerodde (\hat{x}) - A^\dagger\right)\right)}. 
\end{equation}
The expressions for $A, A^\dagger$ in \eqref{eq:A}--\eqref{eq:A_dagger} and the expression for $DF(\hat{x})$ in \eqref{eq:DF} imply that
\begin{align*}
\pi_N A = \numinv \pi_N, \qquad 
\pi_N A^\dagger = \numder \pi_N, \qquad 
\pi_N DF(\hat{x}) = \pi_N DF(\hat{x}) \pi_N
\end{align*}
so that we can rewrite the first term  on the right hand side of \eqref{eq:z1_dde_step1} as
\begin{equation} \label{eq:z1_dde_step2}
\begin{aligned}
\norm*{\pi_N \left(A \left(D\zerodde (\hat{x}) - A^\dagger\right)\right)} &= \norm*{\numinv \pi_N\left(D\zerodde (\hat{x}) - A^\dagger\right)} \\
& = \norm*{\numinv \left( \pi_N D\zerodde (\hat{x}) \pi_N - \numder \right)}.
\end{aligned}
\end{equation}
To bound the second term in the right hand side of \eqref{eq:z1_dde_step1}, we fix a sequence $v \in \mathcal{S}_2$ and 
compute that
\begin{align*}
\left(\pi_\infty \left(A \left(D \zerodde (\hat{x}) - A^\dagger\right)\right)v\right)_\beta  &= \left( \pi_\infty  A D \zerodde (\hat{x}) v \right)_\beta - \left( \pi_\infty A A^\dagger v \right)_\beta \\
&= \Delta^{-1} ( \dotp{\lambda, \beta}) \left[ \Delta(\dotp{\lambda, \beta}) v_\beta + \alpha( \blowupscalardde (\hat{x})*v + \blowupscalardde (v)*\hat{x})_\beta \right] - \Delta^{-1}(\dotp{\lambda, \beta}) \Delta(\dotp{\lambda, \beta}) v_\beta \\
& = \alpha \Delta^{-1} ( \dotp{\lambda, \beta}) (\blowupscalardde (\hat{x})*v + \blowupscalardde (v)*\hat{x})_\beta. 
\end{align*}
This implies that
\begin{equation*}
\begin{aligned}
\norm{\pi_\infty \left(A \left(D\zerodde(\hat{x}) - A^\dagger\right)\right)v} & = 
\sum_{\abs*{\beta} > N} \abs*{\alpha \Delta\left(\langle \lambda, \beta \rangle \right)^{-1} \left(\blowupscalardde (\hat{x})*v + \blowupscalardde(v)*\hat{x}\right)_\beta} \\
& \leq \alpha \sup_{\abs*{\beta} > N} \setof*{\abs*{\Delta(\langle \lambda, \beta \rangle)^{-1}}} \left(\sum_{\abs*{\beta} > N}\abs*{(\blowupscalardde(\hat{x})*v)_\beta} + \abs*{(\blowupscalardde(v)*\hat{x})_\beta}\right) \\
& \leq \alpha \sup_{\abs*{\beta} > N} \setof*{\abs*{\Delta(\langle \lambda, \beta \rangle)^{-1}}} \left(\norm*{\blowupscalardde(\hat{x})*v} + \norm*{\blowupscalardde(v)*\hat{x}}\right) \\
&  \leq \alpha \sup_{\abs*{\beta} > N} \setof*{\abs*{\Delta(\langle \lambda, \beta \rangle)^{-1}}} \left(\norm*{\blowupscalardde(\hat{x})}\norm{v} + \norm*{\blowupscalardde(v)}\norm*{\hat{x}}\right) 
\end{aligned}
\end{equation*}
which together with \eqref{eq:estimate_delta_inverse} gives that 
\begin{equation} \label{eq:z1_dde_step3}
\norm{\pi^\infty \left(A \left(D\zerodde(\hat{x}) - A^\dagger\right)\right)v} \leq \frac{\alpha}{\Re(\lambda)N - \alpha e^{-\Re(\lambda)N}} \left(\norm*{\blowupscalardde(\hat{x})}\norm{v} + \norm*{\blowupscalardde(v)}\norm*{\hat{x}}\right) . 
\end{equation}
Additionally, since $\Re(\lambda) > 0$, the estimate 
\begin{align*}
\norm{\blowupscalardde (v)} &\leq \sum_{\beta \in \mathbb{N}^2} \abs{ e^{-\dotp{\lambda, \beta}} v_\beta} \\
&\leq \sum_{\beta \in \mathbb{N}^2} \abs{v_\beta}
\end{align*}
implies that the operator norm of $\blowupscalardde$ is bounded by 
\begin{equation} \label{eq:r_estimate}
\norm{\blowupscalardde} \leq 1.
\end{equation}
Together with \eqref{eq:z1_dde_step3}, this means that we can estimate the operator norm of $\pi_\infty \left(A \left(D\zerodde(\hat{x}) - A^\dagger\right)\right)$ as 
\begin{equation}
\norm{\pi^\infty \left(A \left(D\zerodde(\hat{x}) - A^\dagger\right)\right)v} \leq \frac{2 \alpha}{\Re(\lambda) N- \alpha e^{-\Re(\lambda)N}} \norm{\hat{x}}. 
\end{equation}
Combining this with \eqref{eq:z1_dde_step1} and \eqref{eq:z1_dde_step2}, we conclude that \eqref{eq:wright_Z1} holds. 
\end{proof}

\begin{lemma}[$Z_2$ bound] \label{lem:Z2_dde}
Let $\zerodde$ be as defined in \eqref{eq:zero_map_wright_dde}, let $\hat{x}  \in \pi_N \cS_2$ be a pseudo-finite sequence and let the operators $A$ be defined via \eqref{eq:A}. Moreover, let $\Re(\lambda)$ be as in \eqref{eq:def_re_lambda}, let $N \in \mathbb{N}$ be satisfying \eqref{eq:N_treshhold_dde} and define the constant
\[ 
Z_2 = 2 \alpha \max \setof*{\norm*{\overbar{A}}, \frac{1}{\Re(\lambda) N- \alpha e^{-\Re(\lambda)N}}}. 
 \]
Then the estimate 
\[
\norm*{A\left(DF(\hat{x}+y) - DF(\hat{x})\right)} \leq Z_2 \norm*{y}. 
\]
holds for all $y \in \mathcal{S}$. 
\end{lemma}
\begin{proof}
We first estimate the operator norm of $A$. To do so, we fix a sequence $v \in \mathcal{S}_2$ and derive from the definition of $A$ in \eqref{eq:A} that 
\begin{align*}
\norm*{Av} & = \sum_{\abs*{\beta} \leq N} \abs*{(\numinv v)_\beta} + \sum_{\abs*{\beta} > N} \abs*{\Delta (\langle \lambda, \beta \rangle )^{-1}v_\beta} \\
& \leq \norm*{\numinv} \sum_{\abs*{\beta} \leq N} \abs*{v_\beta} + \sup_{\abs*{\beta} > N} \setof*{\Delta(\langle \lambda, \beta \rangle )^{-1}} \sum_{\abs*{\beta} > N} \abs*{v_\beta} \\
& \leq \max \setof*{\norm*{\numinv}, \frac{1}{\Re(\lambda) N - \alpha e^{-\Re(\lambda) N}}} \norm*{v}
\end{align*}
where in the last step we used \eqref{eq:estimate_delta_inverse}. Hence we can estimate the operator norm of $A$ as
\begin{equation} \label{eq:estimate_A}
\norm*{A} \leq \max \setof*{\norm*{\numinv}, \frac{1}{\Re(\lambda) N- \alpha e^{-\Re(\lambda)N}}}.
\end{equation}
We next estimate the operator norm of $D\zerodde (\hat{x}+y) - D\zerodde (\hat{x})$. For a fixed sequence $v \in \mathcal{S}_2$, we compute that
\begin{align*}
(DF(\hat{x}+y)v - DF(\hat{x})v)_\beta = \begin{cases}
0 \qquad &\mbox{if } | \beta | \leq 1 \\
\alpha( \blowupscalardde (y)*v + \blowupscalardde (v)*y)_\beta \qquad &\mbox{if } \left| \beta \right| \geq 2
\end{cases}
\end{align*}
which implies that 
\begin{align*}
\norm{\left(D\zerodde (\hat{x}+y) - D\zerodde (\hat{x})\right)v} &\leq | \alpha | (\norm{\blowupscalardde} \norm{y} \norm{v} + \norm{y} \norm{\blowupscalardde}) \norm{v} \\
& = 2 \alpha \norm{r} \norm{y} \norm{v}. 
\end{align*}
Together with \eqref{eq:r_estimate} this yields that 
\[
\norm{D\zerodde (\hat{x}+y) - D\zerodde (\hat{x})} \leq 2 \alpha \norm{y}. 
\]
By combining the above inequality with the estimate \eqref{eq:estimate_A}, the statement of the lemma follows. 
\end{proof}

In the accompanying code on \github, we fix parameters $\alpha = \placealpha$, $n = \placen$, and let $\lambda$ be as in Lemma \ref{lem:eigenvalues_dde_explicit} and $\hat{x}$ be as in Remark \ref{rem:xhat}. We explicitly compute interval enclosures of the quantities  
$Y_0, Z_0, Z_1$ and $Z_2$ as defined in Lemma \ref{lem:Y0_dde}--\ref{lem:Z2_dde} 
We subsequently verify that that the polynomial 
\[
p(r) = Z_2 r^2 - (1 - Z_0 - Z_1) r + Y_0
\]
has a zero on the positive real line. From here, the Radii Polynomial Theorem (Theorem \ref{thm:radii_polynomial}) implies that 

\begin{proposition} \label{prop:r_dde}
Fix $\alpha = 2$ and let $\hat{x} \in \mathcal{S}_N$ be as in Remark \ref{rem:xhat}. Then for 
\[
r_{\text{\tt DDE}} = 1.956701163090857 \times 10^{-9}, 
\]
the map $F$ defined in \eqref{eq:zero_map_wright_dde} has a zero on the ball $B(\hat{x}, r_{\text{\tt DDE}})$. 
\end{proposition}

\subsection{Conclusion}

Combining Propositions \ref{prop:r_ps} and \ref{prop:r_dde} with the triangle inequality yields that

\begin{theorem} \label{thm:distance_conclusion}
Let $\alpha = 2$ and $n = \placen$; let $\lambda_{n} = (\lambda_{n, +}, \lambda_{n, -})$ be as in Lemma \ref{lem:eigenvalues_ps_explicit} and let $\lambda = (\lambda_{+}, \lambda_{-})$ be as in Lemma \ref{lem:eigenvalues_dde_explicit}. 
Then the map $\zerops$ in \eqref{eq:zero_map_wright_ps} has zero $x_{\text{\tt PSA}}$ and the map $\zerodde$ in \eqref{eq:zero_map_wright_dde} has a zero $x_{\text{\tt DDE}}$ that satisfy
\[
\norm{x_{\text{\tt PSA}} - x_{\text{\tt DDE}} } \leq r_{\text{\tt DDE}} + r_{\text{\tt PSA}} \leq 1.956812219647396 \times 10^{-9}.
\]
\end{theorem}

\appendix

\section{Constructive versions of the Contraction Mapping Theorem}
\label{appendix:contraction_mapping}

To carry out the computer assisted proofs in this work we apply two constructive variations of the Contraction Mapping Theorem which we state below. Both of these variations, tailored for computer assisted proofs, are used to prove the existence of zeros and obtain rigorous error estimates.

The first variation, which we will call the ``Radii Polynomial Theorem'', applies to functions on Banach spaces, and we will use it to construct the positive numbers $r_{\text{\tt PSA}}$ and $r_{\text{\tt DDE}}$ from Step 2 and 3 of our scheme.

\begin{theorem}[Radii Polynomial Theorem, cf. \cite{MR1639986,MR2338393}] \label{thm:radii_polynomial}
Let $X$ and $Y$ be Banach spaces and let $F : X \to Y$ be a Fr\'{e}chet differentiable function. Fix $\hat{x} \in X, A^\dagger \in \cL(X, Y), A \in \cL(Y, X)$ and assume $A$ is invertible. Suppose $Y_0, Z_0, Z_1$ are positive constants and $Z_2 : (0, \infty) \to (0, \infty)$ are given and satisfy the following estimates
\begin{equation} \label{eq:constants_banach_radii}
\begin{aligned}
\norm*{AF(\hat{x})} & \leq Y_0 \\
\norm*{I_X - A A^\dagger} & \leq Z_0 \\
\norm*{A\left(DF(\hat{x}) - A^\dagger\right)} & \leq Z_1 \\
\norm*{A\left(DF(x) - DF(\hat{x}\right))} & \leq Z_2(r)r \qquad \forall x \in \overbar{B_r(\hat{x})}, \forall r \in (0, \infty).
\end{aligned}
\end{equation}
Define the radii polynomial to be 
\[
p(r) := Z_2(r)r^2 - (1 - Z_0 - Z_1)r + Y_0.
\]
If there exists $r_0 > 0$ such that $p(r_0) < 0$, then there exists a unique root of $F$ in $B_{r_0}(\hat{x})$.
\end{theorem}

The second variation  (Theorem \ref{thm:contraction_mapping_scalar}) is essentially the Newton–Kantorovich theorem (c.f.~\cite{yamamoto86}) tailored to computer assisted proof. It applies to functions taking values in the complex numbers and in this work we use it to rigorously solve eigenvalue problems for both DDE and ODE.

\begin{theorem} \label{thm:contraction_mapping_scalar}
Let $U \subseteq \mathbb{C}$ be an open subset of the complex numbers and let $g: U \to U$ be a twice differentiable function on $U$. Assume that $g'$ is non-singular on $U$ and define $F: U \to U$ by the formula
\[
F(z) = z - \frac{g(z)}{g'(z)}.
\]
Fix a complex number $\hat{z}$ and a positive number $r^\ast > 0$ and define the quantities $Y, Z$ and $r_0$ as 
\begin{equation} \label{eq:constant_scalar_radii}
Y  := \abs*{F(\hat{z}) - \hat{z}}, \quad 
Z := \sup \left(\abs*{DF(\overbar{B_{r^*}(\hat{z})})}\right), \quad 
r_0 := \frac{Y}{1-Z}.
\end{equation}
If $Z < 1$ and $0 < r_0 < r^*$ both hold, then there exists a unique root of $g$ contained in $\overbar{B_{r_0}(\hat{z})}$. 
\end{theorem}

\begin{proof}
The proof is application of the Contraction Mapping Theorem. We first prove that 
\[F\left(\overbar{B_{r_0}(\hat{z})}\right) \subseteq \overbar{B_{r_0}(\hat{z})}.\]
To that end, we rewrite the definition of $r_0$ to obtain
\begin{equation} \label{eq:r0}
Z \cdot r_0 + Y = r_0. 
\end{equation}
We next pick $z \in \overbar{B_{r_0}(\hat{z})}$ and estimate that
\begin{align*}
\abs*{F(z) - \hat{z}} & \leq \abs*{F(z) - F(\hat{z})} + \abs*{F(\hat{z}) - \hat{z}}\\
& \leq \sup_{z \in \overbar{B_{r_0}(\hat{z})}} \setof*{\abs*{DF(z)}} \abs*{z - \hat{z}} + \abs*{F(\hat{z}) - \hat{z}} \\
& \leq Zr_0 + Y \\
& = r_0
\end{align*}
where in the last step we used \eqref{eq:r0}.  This proves that that $F\left(\overbar{B_{r_0}(\hat{z})}\right) \subseteq \overbar{B_{r_0}(\hat{z})}$. 

We next prove that $F$ is a uniform contraction on $F\left(\overbar{B_{r_0}(\hat{z})}\right)$ with contraction rate $Z < 1$. To that end, we let $z_1, z_2 \in \overbar{B_{r_0}(\hat{z})}$ and estimate
\[
\abs*{F(z_1) - F(z_2)} \leq \sup_{z \in \overbar{B_{r_0}(\hat{z})}} \setof*{\abs*{DF(z)}} \abs*{z_1 - z_2} \leq Z \abs*{z_1 - z_2}.
\] 
Hence $F$ is a uniform contraction on $F\left(\overbar{B_{r_0}(\hat{z})}\right)$ with contraction rate $Z < 1$ and the Contraction Mapping Theorem implies the result.
\end{proof}

\bibliographystyle{plain}
\bibliography{references}

\end{document}